[1994/12/01]
\documentclass[draft]{article}

\title{
Gromov's Problem: Bound the Expansion Coefficient from below in terms of the Observable Diameter 
of
a Metric Measure Space, 
and its Diameter Bounds
}
\author{Ushio Tanaka\footnote{Department of Mathematics and Information Sciences, Osaka Prefecture University, Osaka 599-8531, Japan. 
\newline
{\itshape E-mail address}: utanaka@mi.s.osakafu-u.ac.jp. 
\newline
The author 
was
supported 
by Grant-in-Aid 
for Young Scientists (B), Japan Society for the Promotion of Science(JSPS), Grant Number 25730022.
}
}
\date{\today}

\usepackage{amsmath,amsthm,amssymb}

\chardef\bslash=`\\ 
\newcommand{\ntt}{\normalfont\ttfamily}

\newtheorem{thm}{Theorem}[section]
\newtheorem{cor}[thm]{Corollary}

\newtheorem{prop}[thm]{Proposition}

\newtheorem{example}[thm]{Example}

\theoremstyle{definition}
\newtheorem{defn}{Definition}[section]

\theoremstyle{remark}
\newtheorem{rem}{Remark}[section]

\newtheorem{exercise}{Exercise}
\newtheorem{claim}{Claim}

\newcommand{\R}{\mathbb{R}}
\newcommand\N{\mathbb{N}}

\DeclareMathOperator{\Ric}{\operatorname{Ric}}
\DeclareMathOperator{\ObsDiam}{\operatorname{ObsDiam}}
\DeclareMathOperator{\PartDiam}{\operatorname{PartDiam}}
\DeclareMathOperator{\Expansion}{\operatorname{Exp}}
\DeclareMathOperator{\Var}{\operatorname{Var}}
\DeclareMathOperator{\Lap}{\operatorname{E}}
\DeclareMathOperator{\diam}{\operatorname{diam}}
\DeclareMathOperator{\Sphere}{\mathbb{S}}
\DeclareMathOperator{\Gromov}{\operatorname{\text{Gromov}}}
\DeclareMathOperator{\Ledoux}{\operatorname{\text{Ledoux}}}
\newcommand{\eval}[2][\right]{\relax
  \ifx#1\right\relax \left.\fi#2#1\rvert}

\begin{document}
\maketitle
\markboth{Sample paper for the {\protect\ntt\lowercase{amsmath}} package}
{Sample paper for the {\protect\ntt\lowercase{amsmath}} package}
\renewcommand{\sectionmark}[1]{}
\begin{abstract}
In the celebrated 
book entitled
`Metric Structures for Riemannian and Non-Riemannian Spaces', 
so-called `Green Book', Gromov presented a problem regarding a metric measure space. Gromov posed the question `Bound the expansion coefficient from below in terms of the observable diameter'. 
The overall aim of the current study is to 
demonstrate
the answer to this problem. To 
begin solving 
this problem, 
the concentration of measure phenomenon on 
the metric measure space must be considered. The concentration function to evaluate the measure phenomenon 
is connected by
the observable diameter and the expansion coefficient. Furthermore, the procedure for our answer gives us the upper bound for the expansion coefficient in terms of the observable diameter. 
Combining the desired lower bound for the expansion coefficient with 
its
upper bound, 
we eventually obtain the upper bound for the observable diameter.
Simultaneously, 
this reasoning has enabled
us to obtain the upper bound for the diameter of a bounded metric measure space in terms of the expansion coefficient. 
We 
will
apply the above-mentioned results to a compact connected Riemannian manifold with non-negative Ricci curvature, which makes the bounds 
more explicit. More precisely, they are in terms of the doubling constant of the Riemannian measure and the first non-trivial eigenvalue of the Laplacian on the Riemannian manifold.   
\end{abstract}
\tableofcontents
\section{Introduction}
%
\label{sec:Introduction}
%
%
%
Mikhail Gromov has, in his celebrated 
`Green Book', proposed the following problem, which concerns the expansion coefficient and observable diameter of a metric measure space:
\begin{exercise}[3$\frac{1}{2}$.35 of \cite{RefGromov2001}]
\label{exercise:Gromov}
Bound the expansion coefficient from below in terms of the observable diameter. 
\end{exercise}
%
This paper aims to 
obtain
an answer to 
the above problem
and determine
the novel bounds for the diameter of a bounded metric measure space.
More precisely, we 
will
show that the upper and lower bounds for the diameter are in terms of the expansion coefficient and 
Laplace functional, respectively.

To 
the
best 
of
our knowledge, the expansion coefficient has two proposals: 
one by
Mikhail Gromov and 
the other by Michel Ledoux. 
Needless to say, the expansion coefficient 
stated in Gromov's problem 
has been proposed by
Gromov.  
However,
in the procedure for our calculations, we 
will
exploit Ledoux's expansion coefficient. Therefore, our answer to Gromov's problem is in terms of not only the observable diameter but also Ledoux's expansion coefficient (Theorem~\ref{thm:answer to Gromov}). 

When we apply our answer to a compact connected Riemannian manifold, Ledoux's expansion coefficient makes the lower bound for Gromov's expansion coefficient of the manifold explicit. Explicitly, its lower bound is in terms of the doubling constant of the Riemannian measure and the first non-trivial eigenvalue of the Laplacian on the Riemannian manifold, as well as the observable diameter    (Example~\ref{ex:example of Gromov problem}).      

In order to undertake
Gromov's problem, 
we have paid attention to the concentration of measure phenomenon on a metric measure space. The concentration function enables us to describe the concentration phenomenon. For this reason, previous studies of the concentration phenomenon on the metric measure space have 
centred on 
evaluation of
the concentration function.
Remarkably, the concentration function plays an implicit, 
although pivotal, 
role to bridge 
Gromov's and Ledoux's expansion coefficients and the observable diameter. 

Furthermore, the procedure (Theorems~\ref{thm:bound for diam by Ledoux} and \ref{thm:evaluation of concentration function using Gromov expansion coefficient}) for our answer 
yields additional
by-products. 
These include,
for a metric measure space, 
the lower bound for Ledoux's expansion coefficient of the Riemannian manifold (Corollary~\ref{cor:lower bound for Ledoux's expansion coefficient}), 
the upper bounds for Gromov's expansion coefficient (Corollary~\ref{cor:upper bound for expansion coefficient by Gromov}) and for the observable diameter (Corollary~\ref{cor:upper bound by Ledoux's expansion coefficient for the observable diameter}) and its application to the Riemannian manifold (Example~\ref{ex:result on the upper bounds for observable diameter}) 
and the upper bound for the diameter (Theorem~\ref{thm:bound for diam by Gromov and Ledoux}) and its application to the Riemannian manifold (Example~\ref{ex:upper bound for the diameter of a Riemannian manifold}).

\subsection*{Acknowledgements}
The author would like to thank Shin-ichi Ohta, an associate professor of Kyoto University: Takumi Yokota, an assistant professor of Research Institute for Mathematical Sciences: Dr's.~Yu Kitabeppu and Ryunosuke Ozawa and Mr.~Kohei Suzuki of Kyoto University for their fruitful comments and suggestions. 

\section{Review of materials}
\label{sec:Review of materials}
In this section, for the convenience of the reader, we will briefly summarize some pre-requisite materials of the metric measure spaces for a statement on Gromov's problem and its answer to the problem: the concentration function, the expansion coefficient and the observable diameter. 
It is the crux of the problem to observe the relation between these concepts. The concentration function is a device to connect the expansion coefficient with the observable diameter. These three materials play a significant role in grasping the structure of ambient sources, namely metric measure spaces, when accompanying the concentration of the measure phenomenon. See \cite{RefLedoux2001} and the references therein for further accounts. 


\subsection{Concentration function}
\label{subsec:concentration}
The concentration function on the metric measure spaces describes the concentration of the measure phenomenon on these said spaces. Loosely speaking, the concentration phenomenon occurs 
when a set of 
spaces has 
a
sufficiently large measure, where `most' of the points in the spaces get 
`close' 
to the neighbourhood of the set, which is referred to as an isoperimetric enlargement (called isoperimetric neighbourhood as well; see Definition~\ref{def:isoperimetric enlargement}). The classical isoperimetric inequality in Euclidean space is in terms of isoperimetric enlargement (see \cite[p.~170]{RefLedoux1996} for detailed accounts). 

This
notion 
of utilizing
the concentration function to evaluate the concentration of 
the
measure phenomenon was first introduced in \cite{RefAmir:Milman1980}, which has been formalized in \cite{RefGromov:Milman1983} and further 
analysed
in \cite{RefMilman:Schechtman1986}.
The function relies on two main 
entities: 
a normalized measure
(i.e. probability measure)
and a notion concerning isoperimetric enlargement, with respect to which concentration is evaluated. 

The most significant outcome
for the function is that the measure permits a very small concentration function as the isoperimetric enlargement tends to be `large'. As mentioned in \cite[Section~4]{RefMilman1988}, the chief problem in the investigation of the concentration phenomenon is estimating the concentration function; accordingly Milman shows three techniques including an isoperimetric inequalities approach. It is remarkable that the concentration function may be controlled in a rather large number of cases, especially its central classes---which are those 
that
decay exponentially or Gaussian, which we call exponential concentration or Gaussian (or normal) concentration. We refer the reader to \cite{RefLedoux1996}, \cite{RefLedoux1999} and \cite{RefLedoux2001} for detailed accounts.

\subsection{Observable diameter}
\label{subsec:observable diameter}
The current expository accounts for the observable diameter 
are
principally due to \cite[p.~336]{RefBerger2000} and \cite[Section~1.4]{RefLedoux2001}.

The idea of the observable diameter is to introduce notions corresponding to physical reality and physical experiments, 
and
as mentioned in 
\cite[p.~50]{RefFunano2007} and \cite[p.~105]{RefFunano2010}, due to the quantum and statistical mechanics. Physical reality is 
defined as
a metric space. An object can be observed only by the signals we
can perceive. 
The signals are Lipschitz functions. 
What we perceive, due to the lack of accuracy of our instruments, 
results in
only 
a small error, and the observable diameter is intended to capture this variability. 

The notion concerning the observable diameter can be defined for any geometric concept such as the central radius (the minimal radius of 
a
ball covering the whole metric space)
and
the centre of mass (or barycenter).
A metric and a measure are enough to define such notions. 

Historically, the first estimation of the observable diameter was for standard spheres. As early as 1919, Paul L\'{e}vy studied the concentration phenomenon for the spheres, which 
could
be
described
in terms of the observable diameter; see Example~\ref{ex:Observable diameter by Paul Levy} below for detailed accounts. It is true for Gromov's result for Riemannian manifolds with positive Ricci curvature as well; see Example~\ref{ex:Observable diameter of Riemannian manifold} below for detailed accounts. The contemporary treatment of the observable diameter by Gromov is found to be a `visual' description of the concentration phenomenon. 

The observable diameter may be viewed as 
a dual component of
the concentration function. It describes the diameter of a metric space viewed through a given probability measure on the Borel sets of metric space. 

\subsection{Expansion coefficient}
The notion concerning the expansion coefficient implies a volume ratio of a set of metric measure spaces to its isoperimetric enlargement. The expansion coefficient has two proponents: Gromov (see \cite[Section~{$3\frac{1}{2}.35$}]{RefGromov2001}) and Ledoux (see \cite[Section~1.5]{RefLedoux2001}); accordingly, in the present paper, each of these will be referred to as the
``Gromov's and Ledoux's expansion coefficients,'' respectively.
These two expansion coefficients may be
thought of as working side-by-side. 

It can be inferred
that Ledoux's expansion coefficient is analogous to the so-called \textit{Cheeger isoperimetric constant};
see
\cite[p.32]{RefLedoux2001} and the references therein for detailed accounts. As we will discuss, it is a conspicuous property that when Ledoux's expansion coefficient 
is greater than 1, it gives rise to the above-mentioned exponential concentration.

\section{Concentration of measure phenomenon on metric measure spaces}
\label{sec:Concentration phenomenon}
In this section, using the concentration function, we 
will
work over 
the
concentration of 
the
measure phenomenon on metric measure spaces.  
\subsection{Setup}
We now define the metric measure space in the sense of \cite[Section~1.2]{RefLedoux2001};
also 
refer to
\cite{RefGromov2001} for the pioneering work on the present topic by 
Gromov.
\begin{defn}
A \textit{metric measure space} is defined to be a metric space 
$ (X, d_X) $
equipped with a finite Borel measure $\mu_X$ on 
$(X, d_X)$. We denote it by a triplet $(X, d_X, \mu_X)$, which is often referred to as \textit{mm space} as well. 
\end{defn}
In all what follows, we 
will focus on
$\mu_X (X) = 1$, which permits a stochastic structure of $(X, d_X, \mu_X)$. 
We will
employ the same letter $X$ briefly to designate $(X, d_X, \mu_X)$. 

Let us
now discuss the concentration phenomenon on the metric measure space.
Historically, the concentration of the measure phenomenon was most vigorously put forward by V.D.~Milman in the local theory of Banach spaces in the study of Dvoretzky's theorem on almost Euclidean sections of convex bodies \cite[p.~178]{RefLedoux1996} and 
its
references 
and \cite[p.~14]{RefLedoux:Talagrand1991}. To oscillate the measure on the space dynamically makes it 
capable of
grasping
the 
space
structure.
The dynamic association is exactly the \textit{concentration of the measure phenomenon} on the space. In fact, 
the pertinence of the appellation for 
the concentration phenomenon is due to a concentration inequality on the space; see Proposition~\ref{prop:concentration inequality} and Remark \ref{rem:concentration phenomenon} below.  

As mentioned in \cite[p. 170]{RefLedoux1996}, the concentration phenomenon is rather concerned with the behaviour of `large' isoperimetric enlargement. In fact, the so-called \textit{L\'{e}vy-Gromov isoperimetric inequality} (see Theorem~\ref{thm:Levy-Gromov isoperimetic inequality}) in which the isoperimetric enlargement has a large measure establishes the concentration function. 
\begin{defn}[Isoperimetric enlargement, isoperimetric neighbourhood]
\label{def:isoperimetric enlargement} 
For 
all
$A \subset X$ and for 
all
$r \ge 0$, we define the \textit{isoperimetic enlargement} or 
\textit{isoperimetric neighbourhood} of order $r$ to be
\begin{equation*}
A_r = \{\, x \in X; d_X (x, A) \le r \,\},
\end{equation*}
which is referred to as the $r$-inflation or 
$r$-extension of $A$ with respect to $d_X$ as well.
\end{defn}

\subsection{The L\'{e}vy-Gromov isoperimetric inequality inspires the concentration function}
\label{subsec:Levy-Gromov}
This subsection aims
to find the so-called \textit{L\'{e}vy-Gromov isoperimetric inequality}, 
which is a generalization of L\'{e}vy's isoperimetric inequality; see
\cite[Theorem~2.3]{RefShioya2016} for detailed 
account.
We refer the reader to \cite[Appedix~C]{RefGromov2001} for the L\'{e}vy-Gromov isoperimetric inequality as well as its original issue \cite{RefGromov1980} by 
Gromov.
\begin{thm}[L\'{e}vy-Gromov isoperimetric inequality; e.g., p.~362 of \cite{RefLedoux1990}]
\label{thm:Levy-Gromov isoperimetic inequality}
Let $M$ be an $n \, (\ge 2)$-dimensional compact connected
Riemannian manifold with Ricci curvature $\Ric (M)$ bounded below in terms of
a positive constant. 
Let $\delta > 0$ be a radius of the $n$-dimensional Euclidean sphere ${\Sphere}^n (\delta)$ relative to its intrinsic Riemannian metric such that 
\begin{equation}
\Ric (M) = \Ric ({\Sphere}^n (\delta))\, (= (n-1)/{\delta}^2).
\label{eq:Ricci curvature condition}
\end{equation}
Write $\sigma_{\delta}^n$ for a normalized rotation-invariant measure, namely a normalized Haar measure on ${\Sphere}^n (\delta)$. 
Then, for 
all 
Borel subset $A$ in $M$ and for 
all 
$r > 0$,
\begin{equation*}
\mu_M (A_r) \ge \sigma_{\delta}^n (B_r),
\end{equation*}
where $B$ is a spherical cap of ${\Sphere}^n (\delta)$ such that $\mu_M (A) = \sigma_{\delta}^n (B)$.
\end{thm}

\begin{cor}
\label{cor:Levy}
In particular, if $A \subset M$ has a sufficiently large measure, say $\mu_M (A) \ge 1/2$, then we have, by L\'{e}vy-Gromov isoperimetric inequality, 
\begin{equation}
\mu_M (A_r) \ge 1 - \sqrt{\pi / 8} \exp ( - (n - 1) r^2 / 2 {\delta}^2 ).
\label{eq:Levy-Gromov isoperimetric inequality}
\end{equation}
See \cite{RefMilman1971} for detailed 
account. 

Combining \eqref{eq:Levy-Gromov isoperimetric inequality} with \eqref{eq:Ricci curvature condition} 
yields
\begin{equation}
\mu_M (A_r) \ge 1 - \sqrt{\pi / 8} \exp ( - \Ric (M) r^2 / 2 ).
\label{eq:approx of Levy-Gromov inequality}
\end{equation}
\end{cor}
\begin{rem}
We employ the constant 
`$\sqrt{\pi / 8}$', 
which modifies 
`$\sqrt 2$'
that 
appeared in the original article \cite[p.~844]{RefGromov:Milman1983}; 
\cite[eq.\ (2)]{RefLedoux1990}. See \cite{RefLedoux1990} for the 
account
on the Ricci curvature condition.
Historically, Corollary~\ref{cor:Levy} goes back as far as the work of 
L\'{e}vy \cite{RefLevy1951};
also refer to Poincar\'{e} \cite{RefPoincare1912}.  
A simple and heuristic proof of the result of Corollary~\ref{cor:Levy} by 
L\'{e}vy
is given by Ledoux \cite{RefLedoux1992}, to which we refer the reader for the argument due to the heat semigroup and Bochner's formula.
\end{rem}
Put this way, henceforth, we 
will
regard the sphere and Rimannian manifold as a metric measure space naturally.  

Due to the result in
\eqref{eq:approx of Levy-Gromov inequality}, 
the measure of the complement of $A_r$ is bounded above by the Gaussian kernel (Gaussian density) with the Ricci curvature as a coefficient, which will be referred to as the Gaussian (or normal) concentration; see Definition~\ref{def:Gaussian concentration} below. The Gaussian concentration is a significant class of the concentration phenomenon. Therefore, the L\'{e}vy-Gromov isoperimetric inequality 
enables
us to inspire the concept of the concentration function,
which
describes the concentration phenomenon; see Definitions~\ref{def:concentration function} and \ref{def:concentration function with epsilon} below.

As a result,
the concentration of 
the
measure phenomenon for a metric measure space is concerned with two main 
components:
a finite measure, 
such as a 
probability measure on metric measure spaces, and 
the
above-mentioned isoperimetric enlargement,
with respect to which the measure concentration is evaluated.

\begin{defn}[e.g., Section~1.2 of \cite{RefLedoux2001}]
\label{def:concentration function}
We define the \textit{concentration function} ${\alpha}_{( X, d_X, \mu_X )}$ of metric measure spaces $( X, d_X, \mu_X )$ by 
\begin{equation*}
{\alpha}_{(X, d_X, \mu_X )} (r)
:= 
\\
\sup \{\,1 - \mu_X (A_r); X \supset A: \text{Borel set}, \mu_{X} (A) \ge 1/2 \,\} 
\end{equation*}
for 
all
$r \ge 0$.
\end{defn}
In what follows, 
to exhibit an answer to Gromov's problem, we are mostly concerned with 
the
generalization of Definition~\ref{def:concentration function} 
with respect to the lower bound for $A$ as above; see also \cite[Section~1.3]{RefLedoux2001} for the generalization.
\begin{defn}
\label{def:concentration function with epsilon}
Let $0 < \varepsilon < 1$. 
We define the \textit{concentration function} ${{\alpha}^{\varepsilon}}_{( X, d_X, \mu_X )}$
\begin{equation*}
{{\alpha}^{\varepsilon}}_{( X, d_X, \mu_X )} (r) := \sup \{\,1 - \mu_X (A_r); X \supset A: \text{Borel set}, \mu_{X} (A) \ge \varepsilon \,\}
\end{equation*}
for 
all
$r \ge 0$.
\end{defn}
%
Therefore, 
we see that
\begin{prop}
\label{prop:concentration function}
\begin{equation*}
{{\alpha}^{\varepsilon}}_{( X, d_X, \mu_X )} (r) \le {{\alpha}^{1 - \varepsilon}}_{( X, d_X, \mu_X )} (r) \quad 
\mbox{for all $r \ge 0$},
\end{equation*}
provided $\varepsilon \ge 1/2$, and vice versa.
\end{prop}
Since the 
concept
of the concentration phenomenon is attributed to 
the isoperimetric inequality, 
as described above, the concentration function is also referred to as 
an
`isoperimetric constant'.
For further references in this paper,
we 
will
write ${{\alpha}^{\varepsilon}}_{( X, d_X, \mu_X )}$ with $\varepsilon = 1/2$ as ${\alpha}_{(X, d_X, \mu_X )}$ solely.

As mentioned in Subsection~\ref{subsec:concentration}, 
two significant classes of metric measure spaces share the exponential and 
Gaussian upper bounds for the concentration function, 
each of which
is defined as 
follows:
\begin{defn}[cf., \cite{RefLedoux2001}, especially Section~1.2]
\label{def:Exponential concentration}
Let $0 < \varepsilon < 1$. 
A metric measure space $( X, d_X, \mu_X )$ has \textit{exponential concentration} if there exist 
universal numeric
constants $C_i$, $i = 1, 2$ such that   
\begin{equation}
{{\alpha}^{\varepsilon}}_{( X, d_X, \mu_X )} (r) \le C_1 \exp (- C_2 r) \quad 
\mbox{for all $r \ge 0$}.
\label{eq:exponential concentration}
\end{equation}
\end{defn}
%
Milman has obtained necessary and sufficient conditions for Cheeger's isoperimetric and Poincar\'{e} inequalities on a metric measure space in terms of the exponential concentration;
see \cite[Theorem~1.5]{RefMilman2009} for further accounts. 

\begin{defn}[cf., Section~1.2 of \cite{RefLedoux2001}]
\label{def:Gaussian concentration}
Let $0 < \varepsilon < 1$. 
A metric measure space $( X, d_X, \mu_X )$ has \textit{Gaussian} (\textit{normal}) \textit{concentration} if there exist 
universal numeric
constants $C_i$, $i = 1, 2$ such that 
\begin{equation}
{{\alpha}^{\varepsilon}}_{( X, d_X, \mu_X )} (r) \le C_1 \exp (- C_2 r^2) \quad 
\mbox{for all $r \ge 0$}.
\label{eq:Gaussian concentration}
\end{equation}
\end{defn}

\begin{rem}
Gaussian concentration 
yields a sharper
estimation than exponential concentration:
If an arbitrary metric measure space has Gaussian concentration, then it has exponential concentration. 
Indeed, for 
each universal numeric
constant $C_i > 0$, $i = 1, 2$ there exists a constant ${C^{\prime}}_1 > 0$ such that
\begin{equation*}
C_1 \exp (- C_2 r^2) \le {C^{\prime}}_1 \exp (- C_2 r) \quad 
\mbox{for all $r \ge 0$}.
\end{equation*}
\end{rem}

We 
will show
a few examples of the exponential concentration throughout our results; see Theorems~\ref{thm:bound for diam by Ledoux} and \ref{thm:evaluation of concentration function using Gromov expansion coefficient} below. 
As one of the most typical examples of the Gaussian concentration, we give the concentration phenomenon on Euclidean spheres; see Example~\ref{ex:concentration phenomenon on a sphere}. One can deduce that the L\'{e}vy-Gromov isoperimetric inequality in Corollary~\ref{cor:Levy} is in terms of Gaussian concentration:
\begin{example}[e.g., p.~274 of \cite{RefMilman1988} and pp. 362--363 of \cite{RefLedoux1990}]
\label{ex:concentration phenomenon on a sphere}
Let ${\Sphere}^n (\delta)$ be the $n \, (\ge 2)$-dimensional Euclidean sphere of radius $\delta > 0$ equipped with the geodesic distance $d_{{\Sphere}^n (\delta)}$ and the rotation-invariant normalized measure $\mu_{{\Sphere}^n (\delta)}$. Then, $( {\Sphere}^n (\delta), d_{{\Sphere}^n (\delta)}, \mu_{{\Sphere}^n (\delta)} )$ has Gaussian concentration as follows:
\begin{equation*}
\alpha_{( {\Sphere}^n (\delta), d_{{\Sphere}^n (\delta)}, \mu_{{\Sphere}^n (\delta)} )} (r) \le C_1 \exp (- C_2 r^2) \quad 
\mbox{for all $r \ge 0$},
\label{eq:Levy concentration theorem}
\end{equation*}
with $C_1 = \sqrt{\pi / 8}$ and $C_2 = \Ric ({\Sphere}^n (\delta)) / 2 =  (n-1)/2 {\delta}^2$, the latter due to \eqref{eq:Ricci curvature condition}.
\end{example}


Next, we are concerned with the concept concerning the diameter of bounded metric measure spaces. When $X$ is bounded, it is noteworthy for the range 
of
isoperimetric enlargement $r$ in the concentration function to range up to the diameter of $X$, which 
is
denoted by 
\begin{equation*}
\diam (X) := \sup \{\, d_X (x,y); x, y \in X \,\}.
\end{equation*}
In Theorem~\ref{thm:bound for diam by Gromov and Ledoux}, we will address the estimate for $\diam (X)$.

It is conspicuous that the concentration function tends towards 0 as the isoperimetric enlargement is close to $\diam(X)$. As mentioned in \cite{RefLedoux2001}, this, however, will not usually be specified. 
The function result will
decrease rapidly as the enlargement, or the dimension of $X$, is extremely large, and this reflects the concentration phenomenon.
\subsection{Concentration inequality}
In this subsection, we 
will
establish the concentration inequality for ${{\alpha}^{\varepsilon}}_{(X, d_X, \mu_X )}$. 
Please note the
concepts mentioned below.
\begin{defn}
Set a positive real number $\varepsilon < 1$. Let $f$ be a measurable real-valued function on $(X, d_X, \mu_X)$. Define a real number $m_f$ of $f$ for $\mu_X$ such that
\begin{equation*}
\mu_X (\{\, f \le m_f \,\}) \ge \varepsilon, \quad
\mu_X (\{\, f \ge m_f \,\}) \ge 1 - \varepsilon.
\end{equation*}
If one regards $f$ as a random variable, then $m_f$ is referred to as the \textit{quantile of order $\varepsilon$} of $f$ for $\mu_X$ or the \textit{$100 p$ th percentile} of $f$ for $\mu_X$. 
In particular, if $\varepsilon = 1/2$, then $m_f$ exactly coincides with the so-called \textit{L\'{e}vy mean} or \textit{median} of $f$ for $\mu_X$. Note that $m_f$ exists and may not be unique. Nevertheless \cite[p.~21]{RefLedoux:Talagrand1991} shows that the median of the Gaussian kernel (density) for the canonical Gaussian measure on 
an $n$-dimensional Euclidean space
is uniquely determined.  
\end{defn}

As will be shown below, the Lipschitz property on metric spaces $(X, d_X)$, 
involving
Lipschitz function and its Lipschitz constant, 
enables
us to observe the concentration phenomenon on $(X, d_X, \mu_X)$ and to introduce the concept of the observable diameter of $(X, d_X, \mu_X)$; see Section~\ref{sec:Observable diameter}.
\begin{defn}
\label{def:Lipschitz}
We call a real-valued function $f$ on $(X, d_X)$ \textit{Lipschitz} if
\begin{equation*}
\| f \|_{\text{Lip}} := \sup_{x, y \in X; x \neq y} \frac{|f(x) - f(y)|}{d_X (x,y)} < \infty.
\end{equation*}
$\| f \|_{\text{Lip}}$ is referred to as the Lipschitz constant of $f$. In particular, we say that $f$ is 1-Lipschitz if $\| f \|_{\text{Lip}} \le 1$. 
\end{defn}

We are now ready 
to state the concentration inequality for ${{\alpha}^{\varepsilon}}_{(X, d_X, \mu_X )}$: 
\begin{prop}[Concentration inequality]
\label{prop:concentration inequality}
Set a positive real number $\varepsilon < 1$. Let $f$ be a Lipschitz function on $(X, d_X)$, and let $m_f$ be the quantile of order $\varepsilon$.
We have 
\begin{equation}
\mu_X ( \{\, |f - m_f| > r \,\} ) \le {{\alpha}^{\varepsilon}}_{(X, d_X, \mu_X )} (r/\| f \|_{\text{Lip}}) + {{\alpha}^{1 - \varepsilon}}_{(X, d_X, \mu_X )} (r/\| f \|_{\text{Lip}}) 
\label{eq:concentration inequality}
\end{equation}
for 
all
$r \ge 0$.
In particular, if $f$ is 1-Lipschitz, then \eqref{eq:concentration inequality} is given by
\begin{equation}
\mu_X ( \{\, |f - m_f| > r \,\} ) 
\le 
{{\alpha}^{\varepsilon}}_{(X, d_X, \mu_X )} (r) + {{\alpha}^{1 - \varepsilon}}_{(X, d_X, \mu_X )} (r) \quad 
\text{for all $r \ge 0$}.
\label{eq:concentration inequality with 1-Lip}
\end{equation}
\end{prop}

\begin{proof}
Set $A: = \{ f \le m_f \}$. 
Then
$\mu_X (A) \ge \varepsilon$ follows from definition $m_f$. For 
all
$r \ge 0$, fix $x \in A_r$. 
One can see that
\begin{equation}
\mu_X (A_r) \le \mu_X (\{\, f \le m_f + \| f \|_{\text{Lip}} \, r \,\}) \quad 
\text{for all $r \ge 0$}.
\label{eq:equation 1 of concentration inequality}
\end{equation}
Indeed, it follows immediately from Definition~\ref{def:Lipschitz} that 
\begin{equation*}
f(x) \le f(a) + \| f \|_{\text{Lip}} \, d_X (x,a) \quad 
\text{for all $x, a \in X$}.
\end{equation*}
Hence, especially for $a \in A$, we 
actually have $f(x) \le m_f + \| f \|_{\text{Lip}} \, d_X (x,a)$. By taking the infimum over $a \in A$, we have, from the definition of $x \in A_r$, $f(x) \le m_f + \| f \|_{\text{Lip}} \, r$. Therefore, we have $x \in \{\, f \le m_f + \| f \|_{\text{Lip}} \, r \,\}$. This implies \eqref{eq:equation 1 of concentration inequality}, namely
\begin{equation*}
\mu_X (\{\, f > m_f + \| f \|_{\text{Lip}} \, r) \,\}) \le 1 - \mu_X (A_r) 
\quad \text{for all $r \ge 0$}. 
\end{equation*}
Hence,
\begin{equation}
\mu_X (\{\, f > m_f + r \,\}) \le {{\alpha}^{\varepsilon}}_{(X, d_X, \mu_X )} (r/\| f \|_{\text{Lip}}) 
\quad \text{for all $r \ge 0$}. 
\label{eq:equation 2 of concentration inequality}
\end{equation}
We call \eqref{eq:equation 2 of concentration inequality} a \textit{deviation inequality}; see \cite[p.~6]{RefLedoux2001}. 

We now apply this argument again, with $A$ replaced by $\{ -f \le -m_f \}$, to obtain    
\begin{equation*}
\mu_X (A_r) \le \mu_X (\{\, f \ge m_f - \| f \|_{\text{Lip}} \, r \,\}).
\end{equation*}
Similarly,
we can see that
\begin{equation*}
\mu_X (\{\, f < m_f - \| f \|_{\text{Lip}} \, r \,\}) \le 1 - \mu_X (A_r) 
\quad \text{for all $r \ge 0$}. 
\end{equation*}
Hence,
\begin{equation}
\mu_X (\{\, f < m_f - r \,\}) 
\le {{\alpha}^{1 - \varepsilon}}_{(X, d_X, \mu_X )} (r/\| f \|_{\text{Lip}}) 
\quad \text{for all $r \ge 0$}, 
\label{eq:equation 3 of concentration inequality}
\end{equation}
where $1-\varepsilon$ is due to 
$\mu_X (A) = \mu_X (\{ f \ge m_f \}) \ge 1 - \varepsilon$.

We conclude from \eqref{eq:equation 2 of concentration inequality} and \eqref{eq:equation 3 of concentration inequality} that
\begin{equation*}
\mu_X (\{\, |f - m_f| > r \,\}) \le 
{{\alpha}^{\varepsilon}}_{(X, d_X, \mu_X )} (r/\| f \|_{\text{Lip}}) + {{\alpha}^{1 - \varepsilon}}_{(X, d_X, \mu_X )} (r/\| f \|_{\text{Lip}}).
\end{equation*}

In particular, if $\| f \|_{\text{Lip}} \le 1$ in \eqref{eq:concentration inequality}, then we see from the fact that ${{\alpha}^{1 - \varepsilon}}_{(X, d_X, \mu_X )}$ 
decreases such
that 
\begin{equation*}
{{\alpha}^{1 - \varepsilon}}_{(X, d_X, \mu_X )} ( r / \| f \|_{\text{Lip}} ) \le {{\alpha}^{1 - \varepsilon}}_{(X, d_X, \mu_X )} (r) \quad \text{for all $r \ge 0$}.
\end{equation*}
Hence, we obtain \eqref{eq:concentration inequality with 1-Lip}. This proves the proposition.
\end{proof}

\begin{prop}[Concentration inequality]
\label{prop:seaquential concentration inequality}
Under the hypotheses of Propositions~\ref{prop:concentration inequality}, 
we get
\begin{equation}
\mu_X (\{\, |f - m_f| > r \,\}) \le 2 {{\alpha}^{1 - \varepsilon}}_{(X, d_X, \mu_X )} (r/\| f \|_{\text{Lip}}) \quad \mbox{if $\varepsilon \ge 1/2$}.
\label{eq:sequential concentration inequality1}
\end{equation}
In particular, if $f$ is 1-Lipschitz, then \eqref{eq:sequential concentration inequality1} 
is given by
\begin{equation}
\mu_X (\{\, |f - m_f| > r \,\}) \le 2 {{\alpha}^{1 - \varepsilon}}_{(X, d_X, \mu_X )} (r) \quad \mbox{if $\varepsilon \ge 1/2$}.
\label{eq:concentration inequality with 1-Lipschitz}
\end{equation}
\end{prop}

\begin{proof}
Combining \eqref{eq:equation 2 of concentration inequality} and \eqref{eq:equation 3 of concentration inequality} with Proposition~\ref{prop:concentration function}, \eqref{eq:sequential concentration inequality1} 
readily
follows.
The verification for the case that $f$ is 1-Lipschitz coincides with that of Proposition~\ref{prop:concentration inequality}. Thus, we have \eqref{eq:concentration inequality with 1-Lipschitz}. This completes the proof.
\end{proof}

\begin{rem}
\label{rem:concentration phenomenon}
One can see from the two aforementioned concentration inequalities that the Lipschitz function is \textit{concentrated} around its L\'{e}vy mean, with 
the 
rate given by the concentration function.
\end{rem}
%
\section{Observable diameter}
\label{sec:Observable diameter}
In this section, we 
will focus on
the observable diameter of metric measure spaces. As mentioned in \cite[Section~1.4]{RefLedoux2001}, 
the
obserbavle diameter 
might work in conjunction with
the concentration function; see e.g., 
Propositions~\ref{prop:Concentration-ObservableDiameter1} and \ref{prop:Concentration-ObservableDiameter2} below.

\subsection{Partial diameter}
To introduce the observable diameter of metric measure spaces $(X, d_X, \mu_x)$, 
we first need to define the partial diameter. 
For a thorough discussion 
on
the observable diameter, we refer the reader to \cite[$3\frac{1}{2}.20$] {RefGromov2001} and \cite[336--337]{RefBerger2000}. 
\begin{defn}[Partial diameter]
Let $\kappa > 0$. 
We call the infimal $D$ such that there exists a subset $A$ of $X$ with $\diam (A) \le D$ and $\mu_X (A) \ge 1 - \kappa$ the \textit{partial diameter} of $X$ with respect to $\mu_X$. We denote by $\PartDiam_{\mu_{X}} ( X; 1 - \kappa )$ the partial diameter.
\end{defn}
\begin{defn}[Lipschitz dominate; Definition~2.10 of \cite{RefShioya2016}]
Let $(X, d_X, \mu_X)$ and $(Y, d_Y, \mu_Y)$ be 
the
metric measure spaces, respectively. We say that $X$ \textit{Lipschitz dominates} $Y$ if there exists a 1-Lipschitz map $f: X \to Y$ such that
\begin{equation*}
f_{\ast} \mu_{X} = \mu_{Y},
\end{equation*}
where $f_{\ast} \mu_X$ stands for the \textit{push-forward measure} of $\mu_X$ by $f$. 
\end{defn}
The following asserts that the partial diameter is monotone for the aforementioned Lipschitz domination:
\begin{prop}[$3\frac{1}{2}.20$ of \cite{RefGromov2001} and Section~1.4 of \cite{RefLedoux2001}]
Suppose that $X$ Lipschitz dominates $Y$.
Then, it follows readily that
\begin{equation*}
\PartDiam_{\mu_Y} ( Y; 1 - \kappa ) \le \PartDiam_{\mu_X} ( X; 1 - \kappa ).
\end{equation*}
\end{prop}
%
%
\subsection{Observation device for diameter}
What is not obvious 
is that the partial diameter may dramatically decrease under all 1-Lipschitz maps from a metric measure space to a certain metric space. We 
will
now call the target metric space the \textit{screen}; see \cite[$3\frac{1}{2}.20$]{RefGromov2001} and \cite[Section~1.4]{RefLedoux2001}. 
Denoting 
the screen set 
as 
a 1-dimensional Euclidean space $\R$
provides us with more geometric view to concentration. The geometric observation device itself can 
also
be the observable diameter. 
In actuality,
the observable diameter permits us to describe the diameter of a metric measure space viewed through a given Borel probability measure on the space. 

Incidentally, Naor et al. have discussed a class of metric measure space whose observable diameter is much smaller than its diameter, which is sometimes (following Milman) referred to as a ``small isoperimetric constant''; see \cite{RefNaor2005} for detailed accounts.

\begin{defn}[Observable diameter]
\label{def:Observable diameter}
We define the ($\kappa$-)\textit{observable diameter} of $( X, d_X, \mu_X )$ with respect to $\mu_X$, denoted by $\ObsDiam ( X; - \kappa )$, to be the supremum of $\PartDiam_{f_{\ast} \mu_X} ( \R; 1 - \kappa )$ over each 
$f_{\ast} \mu_X$, namely 
\begin{multline*}
\ObsDiam ( X; - \kappa ) 
\\
:= \sup \{ \PartDiam_{f_{\ast} \mu_X} ( \R; 1 - \kappa ); \text{1-Lipschitz function}~f: X \to \R \},
\end{multline*}
where the supremum is taken over all $f$. 
\end{defn}

\begin{rem}
\label{rem:Observable diameter}
According to \cite[p.~336]{RefBerger2000} and \cite[$3\frac{1}{2}$.20]{RefGromov2001}, the observable diameter is usually rather insensitive to 
a positive real number $\kappa < 1$.
Actually, Gromov suggests 
setting
$\kappa = 10^{-10}$. 
Therefore, one may employ the notation $\ObsDiam (X)$ simply for the observable diameter  $\ObsDiam ( X; - \kappa )$.
\end{rem}
%
To facilitate access to the observable diameter, we shall briefly review \cite[pp.~336--337]{RefBerger2000}, \cite[Section~3$\frac{1}{2}$.20]{RefGromov2001} and \cite[Section~1.4]{RefLedoux2001}.
Taken
from a physical point of view, one ascribes the idea of 
observable diameter to the notions corresponding to 
`physical reality' 
and 
`physical experiments'. 
Indeed, we may actually take the physical reality, namely configuration space to be a metric space $( X, d_X )$. 
We think of $\mu_X$ on $( X, d_X )$ as a `state' on the configuration space. 1-Lipschitz functions $f$ on a metric measure space $(X, d_X, \mu_X)$ behave 
like the signals; more precisely, 
`observable', 
namely an observation device giving us the visual (tomographic) image on the screen $\R$; see Corollary \ref{cor:observable diameter}. Thus, with the naked 
eye,
one can view the state via the observable $f_{\ast} \mu_X$ on the screen and cannot identify a part of the screen of measure (luminosity) less than a positive real number $\kappa < 1$. 
\begin{prop}
Suppose that $X$ Lipschitz dominates $Y$. Then, it follows that
\begin{equation*}
\ObsDiam ( Y; - \kappa ) \le \ObsDiam ( X; - \kappa ).
\end{equation*}
\end{prop}
\begin{proof}
See \cite[Proposition~2.18]{RefShioya2016}.
\end{proof}
\subsection{Duality between the concentration function and the observable diameter}
In this subsection, we 
will
discuss the duality between the concentration function and the observable diameter as follows:
%
%
\begin{prop}[Proposition~1.12 of \cite{RefLedoux2001}]
\label{prop:Concentration-ObservableDiameter1}
Let $\kappa > 0$ be small.
Then, we have
\begin{equation*}
\ObsDiam ( X; - \kappa ) \le 2 \inf \{ r > 0; {\alpha}_{(X, d_X, \mu_X )} (r) \le \kappa /2 \}.
\end{equation*}
\end{prop}

The following assertion is still true:
\begin{prop}
\label{prop:Concentration-ObservableDiameter2}
Set a positive real number $\varepsilon < 1$. 
We have 
for a positive real number $\kappa < 1$
\begin{equation}
\ObsDiam ( X; - \kappa ) \le 2 \inf \{ r > 0; {{\alpha}^{\varepsilon}}_{(X, d_X, \mu_X )} (r) + {{\alpha}^{1 - \varepsilon}}_{(X, d_X, \mu_X )} (r) \le \kappa \}.
\label{eq:upper bound for the observable diameter}
\end{equation}
In particular,
\begin{equation}
\ObsDiam ( X; - \kappa ) \le 2 \inf \{ r > 0; {{\alpha}^{\varepsilon}}_{(X, d_X, \mu_X )} (r) \le \kappa/2 \}, \quad \varepsilon \le 1/2. 
\label{eq:sequential Concentration-ObservableDiameter}
\end{equation}
\end{prop}
The argument of the current proof is due to that of Proposition~1.12 of \cite{RefLedoux2001}:
\begin{proof}
For a small $\kappa > 0$, pick $r > 0$ 
such that 
\begin{equation}
{{\alpha}^{\varepsilon}}_{(X, d_X, \mu_X )} (r) + {{\alpha}^{1 - \varepsilon}}_{(X, d_X, \mu_X )} (r) \le \kappa.
\label{eq:kappa condition}
\end{equation}
Let $f$ be a 1-Lipschitz function on $X$. Set $A := f ( \{\, x \in X; |f(x) - m_f| \le r \,\} )$, where $m_f$ is the quantile of order $\varepsilon$ of $f$ for $\mu_X$, 
i.e. it
satisfies $\mu_X ( \{\, f \le m_f \,\} ) \ge \varepsilon$ and $\mu_X ( \{\, f \ge m_f \,\} ) \ge 1 - \varepsilon$. 
To see the observable diameter, observe that
\begin{align*}
f_{\ast} \mu_X (A)
& \ge \mu_X ( \{\, x \in X; |f(x) - m_f| \le r \,\} ) 
\\
& \ge 1 - \left( {{\alpha}^{\varepsilon}}_{(X, d_X, \mu_X )} (r) + {{\alpha}^{1 - \varepsilon}}_{(X, d_X, \mu_X )} (r) \right),
\end{align*}
where we have used \eqref{eq:concentration inequality with 1-Lip} in the last inequality.
Further, from \eqref{eq:kappa condition}, we have
\begin{equation}
f_{\ast} \mu_X (A) \ge 1 - \kappa.
\label{eq:evaluation for pushforward measure}
\end{equation}
Furthermore, it turns out that 
\begin{equation}
\diam (A) 
\le
( m_f + r )  - ( m_f - r ) 
=
2r.
\label{eq:evaluation for diam}
\end{equation}
Thereby, adding \eqref{eq:evaluation for pushforward measure} and \eqref{eq:evaluation for diam}, we obtain
\begin{equation*}
\PartDiam_{f_{\ast} \mu_X} ( \R; 1 - \kappa ) \le 2r,
\end{equation*}
from which \eqref{eq:upper bound for the observable diameter} follows. 

For the remainder of this paper,
on account of Proposition~\ref{prop:concentration function}, we can 
select
$1/2$ as a threshold for $\varepsilon > 0$, 
which
appeared in Definition~\ref{def:concentration function with epsilon}. By combining the aforementioned argument with Proposition~\ref{prop:concentration function}, we eventually see that
\begin{gather*}
\ObsDiam ( X; - \kappa ) \le 2 \inf \{ r > 0; {{\alpha}^{1 - \varepsilon}}_{(X, d_X, \mu_X )} (r) \le \kappa/2 \} \quad \mbox{if $\varepsilon \ge 1/2$};
\\
\ObsDiam ( X; - \kappa ) \le 2 \inf \{ r > 0; {{\alpha}^{\varepsilon}}_{(X, d_X, \mu_X )} (r) \le \kappa/2 \} \quad \mbox{if $\varepsilon \le 1/2$}.
\end{gather*}
In consequence, we obtain \eqref{eq:sequential Concentration-ObservableDiameter} as desired.
\end{proof}

On account of Propositions~\ref{prop:Concentration-ObservableDiameter1} and \ref{prop:Concentration-ObservableDiameter2}, which 
imply
the duality between the concentration function and the observable diameter, the upper bound for the concentration function 
enables
us to control the observable diameter. While the following corollary is fairly straightforward, it plays a crucial role in giving the answer to the current Gromov's problem. 
\begin{cor}
\label{cor:observable diameter}
If $X$ has exponential concentration \eqref{eq:exponential concentration}, 
then we have by 
\eqref{eq:sequential Concentration-ObservableDiameter}
\begin{equation*}
\ObsDiam ( X; - \kappa ) \le \frac{2}{C_2} \ln {\frac{2 C_1}{\kappa}}, \quad \kappa > 0,
\end{equation*}
where each universal numeric constant $C_i > 0$, $i = 1, 2$ has already appeared in \eqref{eq:exponential concentration}.  

If $X$ has Gaussian concentration \eqref{eq:Gaussian concentration}, 
then we have by 
\eqref{eq:sequential Concentration-ObservableDiameter}
\begin{equation*}
\ObsDiam ( X; - \kappa ) \le 2 \sqrt{\frac{1}{C_2} \ln {\frac{2 C_1}{\kappa}}}, \quad \kappa > 0,
\end{equation*}
where each universal numeric constant $C_i > 0$, $i = 1, 2$ has already appeared in \eqref{eq:Gaussian concentration}.  
\end{cor}
\begin{rem}[cf. p.~15 of \cite{RefLedoux2001}]
The significant parameter $C_2$ 
that
appeared in Corollary~\ref{cor:observable diameter} implies the exponential decay of the concentration function.
\end{rem}
Corollary~\ref{cor:observable diameter} 
further
facilitates access to the observable diameter.
Examples follow below:
\begin{example}[cf. Section~1.1 and p.~15 of \cite{RefLedoux2001}]
\label{ex:Observable diameter by Paul Levy}
Combining Corollary~\ref{cor:observable diameter} with Example~\ref{ex:concentration phenomenon on a sphere} shows 
that the observable diameter of the $n\, (\ge 2)$-dimensional Euclidean sphere of radius $\delta$, 
namely
${\Sphere}^n (\delta)$ is of the 
order 
$n^{-1/2}$; more precisely, 
\begin{equation*}
\ObsDiam ( {\Sphere}^n (\delta); - \kappa ) \le 2 \delta \sqrt{\frac{2}{n-1} \ln \sqrt{\frac{\pi}{2}} \frac{1}{\kappa}}, \quad \kappa > 0,
\end{equation*}
from which it follows that
\begin{equation*}
\ObsDiam ( {\Sphere}^n (\delta); - \kappa ) = O ( n^{-1/2} ), \quad n \to \infty. 
\end{equation*}
Specifically, the observable diameter of the unit sphere is given explicitly; see \cite{RefShioya2016} for further accounts.
\end{example}
\begin{example}
\label{ex:Observable diameter of Riemannian manifold}
Let $M$ be an $n \, (\ge 2)$-dimensional compact connected Riemannian manifold with Ricci curvature $\Ric (M) \ge K$ for some constant $K > 0$. 
Applying Corollary~\ref{cor:observable diameter} to Corollary~\ref{cor:Levy}, we get
\begin{equation*}
\ObsDiam ( M; - \kappa ) \le 2 \sqrt{\frac{2 (n-1)}{K n} \ln \sqrt{\frac{\pi}{2}} \frac{1}{\kappa}}, \quad \kappa > 0.
\end{equation*}
Now, in Example~\ref{ex:result on the upper bounds for observable diameter}, we 
will
derive the upper bound for $\ObsDiam ( M; - \kappa )$ in terms of the first non-trivial eigenvalue of the Laplacian on $M$ 
to be stated in Subsection~\ref{subsec:Application to a compact Riemannian manifold} later, 
and the doubling constant of the Riemannian measure of $M$ 
to be stated in Appendix~\ref{appendix:Appendix A} later.
\end{example}
\begin{example}
Let $M$ be a compact connected Riemannian manifold. 
Applying Corollary~\ref{cor:observable diameter} to Theorem~\ref{thm:Gromov-Milman} 
to be shown later, 
we get
\begin{equation*}
\ObsDiam ( M; - \kappa ) \le \frac{2 \ln (3/2 \kappa)}{\ln (3/2) \sqrt{\lambda_1 (M)}}, \quad \kappa > 0.
\end{equation*}
\end{example}

\section{Expansion coefficients}
\label{sec:Expansion Coefficient}
The expansion coefficient of metric measure spaces is proposed by Gromov and Ledoux independently; see \cite[3$\frac{1}{2}$.35]{RefGromov2001} and \cite[Section~1.5]{RefLedoux2001}, respectively. 
Before stating the 
proof
results on the expansion coefficient, 
let us correct their statements of its definition.

Gromov has defined the expansion coefficient to be the infimum of real numbers $e \ge 1$ such that, if $\mu_X (A) \ge \varepsilon$ 
for all 
$A \subset X$, then it follows that $\mu_X (A_{\rho}) \ge e \varepsilon$ for $\rho > 0$.
In contrast,
Ledoux has defined it to be the infimum of real numbers $e \ge 1$ such that, if $\mu_X (B_{\rho}) \le 1/2$ 
for all 
$B \subset X$, then it follows that $\mu_X (B_{\rho}) \ge e \mu_X (B)$ for $\rho > 0$.

In the work that follows,
utilizing the proposal by both
Gromov and Ledoux, we shall distinguish these two expansion coefficients: 
$\Expansion_{\Gromov}$ and $\Expansion_{\Ledoux}$; see below-mentioned Definitions \ref{def:Expansion coefficient by Gromov} and \ref{def:Expansion coefficient by Ledoux} for their definitions.
Nevertheless, $\inf$ in both the above definitions 
is not correct because
$\Expansion_{\Gromov} = \Expansion_{\Ledoux} \equiv 1$. 
Therefore, $\inf$ should be substituted by $\sup$; see Definitions \ref{def:Expansion coefficient by Gromov} and \ref{def:Expansion coefficient by Ledoux}.
\subsection{Gromov's expansion coefficient and its properties}
\begin{defn}[Erratum for Gromov's Expansion coefficient; see 3$\frac{1}{2}$.35 of \cite{RefGromov2001}]
\label{def:Expansion coefficient by Gromov}
Set a positive real number $\varepsilon < 1$. 
We define \textit{Gromov's expansion coefficient} of $\mu_X$ on $(X,d_X)$ of order $\rho > 0$ to be
\begin{equation}
\Expansion_{\Gromov} ( X; \varepsilon, \rho ) := \sup \{e \ge 1; \mu_X (A_{\rho}) \ge e \varepsilon, X \supset A: \text{Borel set}, \mu_X (A) \ge \varepsilon \}.
\label{eq:expansion coefficient by Gromov}
\end{equation}
It turns out 
from \eqref{eq:expansion coefficient by Gromov} that 
\begin{equation}
\mu_X (A_{\rho}) \ge \Expansion_{\Gromov} ( X; \varepsilon, \rho ) \varepsilon.
\label{eq:property of expansion coefficient by Gromov}
\end{equation}
\end{defn}
The following asserts that Gromov's expansion coefficient is monotone for Lipschitz maps:
\begin{prop}[cf. $3\frac{1}{2}$.35 of \cite{RefGromov2001}]
\label{prop:monotone expansion coefficient by Gromov}
Let $f$ be a Lipschitz map between $(X, d_X, \mu_X )$ and $(Y, d_Y, f_{\ast} \mu_{X})$. 
Then, we have
\begin{equation*}
\Expansion_{\Gromov} ( X; \varepsilon, \rho / \| f \|_{\text{Lip}} ) \le \Expansion_{\Gromov} ( Y; \varepsilon, \rho ).
\end{equation*}
In particular, if $X$ Lipschitz dominates $Y$, then
it 
instantly
follows 
that
\begin{equation*}
\Expansion_{\Gromov} ( X; \varepsilon, \rho ) \le \Expansion_{\Gromov} ( Y; \varepsilon, \rho ).
\end{equation*}
\end{prop}
\begin{proof}
The following claim makes it allowable to evaluate the inequalities above. The verification of the claim is straightforward:
\begin{claim}
\label{claim:Lipschitz}
If $f$ is a Lipschitz map, then
\begin{equation*}
A_{r/\| f \|_{\text{Lip}}} \subset f^{-1} ( (f(A))_r ) 
\quad \mbox{for all $r \ge 0$}.
\end{equation*}
\end{claim}
We have
\begin{align*}
&\Expansion_{\Gromov} ( X; \varepsilon, \rho / \| f \|_{\text{Lip}} ) 
\\
& \le \sup \{\, e \ge 1; f_{\ast} \mu_{X} ((f(A))_{\rho}) \ge e \varepsilon, \mu_X (A) \ge \varepsilon \,\}
\quad \mbox{by Claim~\ref{claim:Lipschitz}}
\\
& \le \sup \{\, e \ge 1; f_{\ast} \mu_{X} ((f(A))_{\rho}) \ge e \varepsilon, f_{\ast} \mu_{X} (f(A)) \ge \varepsilon \,\}
\\
& = \Expansion_{\Gromov} ( Y; \varepsilon, \rho ),
\end{align*}
as required.  
\end{proof}
\subsection{Ledoux's expansion coefficient and its properties}
\begin{defn}[Erratum for Ledoux's Expansion coefficient; see Remark~\ref{rem:Ledoux expansion coefficient}]
\label{def:Expansion coefficient by Ledoux}
Set a positive real number $\varepsilon < 1$. 
We define \textit{Ledoux's expansion coefficient} of $\mu_X$ on $(X,d_X)$ of order $\rho > 0$ to be
\begin{multline}
\Expansion_{\Ledoux} ( X; \varepsilon, \rho ) 
\\
:= \sup \{e \ge 1; \mu_X (B_{\rho}) \ge e \mu_X (B), X \supset B: \text{Borel set}, \mu_X (B_{\rho}) \le \varepsilon \}.
\label{eq:expansion coefficient by Ledoux}
\end{multline}
\end{defn}
\begin{rem}
\label{rem:Ledoux expansion coefficient}
Ledoux has originally proposed the expansion coefficient with $\varepsilon = 1/2$; see Section~1.5 of \cite{RefLedoux2001}. As we 
mentioned 
previously
in Section~\ref{sec:Introduction}, 
if $M$ is 
a compact Riemannian manifold, then $\Expansion_{\Ledoux} ( M; 1/2, \rho )$ is analogous to Cheeger's isoperimetric constant; 
see
\cite{RefCheeger1970} for further accounts. In fact, from the viewpoint of expander graphs, Ledoux discusses the relation between Ledoux's expansion coefficient with $\varepsilon = 1/2$ and Cheeger's isoperimetric constant; see \cite[pp.~31--32]{RefLedoux2001} and the reference therein.
\end{rem}
It turns out 
from \eqref{eq:expansion coefficient by Ledoux} 
that
\begin{equation} 
\mu_X (B_{\rho}) \ge \Expansion_{\Ledoux} ( X; \varepsilon, \rho ) \mu_X (B).
\label{eq:expansion coefficient by Ledoux property}
\end{equation}
If $B$ is such that $\mu_X (B_{k \rho}) \le \varepsilon$ for some integer $k \ge 1$, then 
\eqref{eq:expansion coefficient by Ledoux property} 
inductively yields
\begin{equation}
\left( \Expansion_{\Ledoux} ( X; \varepsilon, \rho ) \right)^k \mu_X (B) \le \mu_X (B_{k \rho}) \le \varepsilon. 
\label{eq:iterative Ledoux Expansion coefficient}
\end{equation}
\begin{rem}
One sees immediately from \eqref{eq:iterative Ledoux Expansion coefficient} that if $\Expansion_{\Ledoux} ( X; \varepsilon, \rho ) > 1$, then $B$ has
an
extremely small measure.
In what follows, 
we shall principally concern ourselves with metric measure spaces with 
$\infty > \Expansion_{\Ledoux} ( X; \varepsilon, \rho ) > 1$; 
see Appendix~\ref{appendix:Appendix A} for its observation.
\end{rem}

\subsection{Application to a Riemannian manifold}
\label{subsec:Application to a compact Riemannian manifold}
Throughout this subsection, let $M$ be a compact connected Riemannian manifold and $\varDelta$ the Laplacian (Laplace-Beltrami operator) on $M$. As mentioned in \cite[p.~363]{RefLedoux1990}, what can be said 
if
the concentration function on a compact Riemannian manifold when no lower bound for the Ricci curvature is available
is unclear? 
Concerning this problem, Gromov and Milman have observed the isoperimetric enlargement on a compact connected Riemannian manifold; see \cite[Theorem~4.1]{RefGromov:Milman1983}. Their proof is 
provided by
the so-called Poincar\'{e} inequality:
\begin{thm}[Poincar\'{e} inequality; e.g., \cite{RefMilman1988}]
\label{thm:Poincare inequality with spectral gap}
It is well-known that $- \varDelta$ has its discrete spectrum consisting of the eigenvalues $0 = \lambda_0 < \lambda_1 (M) \le \lambda_2 (M) \dots$. The mini-max principle characterizes the first non-trivial eigenvalue denoted by $\lambda_1(M)$ as the largest constant in the Poincar\'{e} inequality
\begin{equation*}
\lambda_1 (M) \Var_{\mu_M} (f) \le \int_M |\nabla f|^2\,d\mu_M
\end{equation*}
for each smooth real-valued 
function
$f$ on $M$, where $\Var_{\mu_M} (f)$ stands for the variance of $f$ with respect to $\mu_M$, namely 
\begin{equation*}
\Var_{\mu_M} (f):= \int_{M} {\left| f - \int_M f\,d\mu_M \right|}^2 \,d\mu_M,
\end{equation*}
and where $|\nabla f|$ stands for the Riemannian length of the gradient of $f$.
\end{thm}

Now, using the argument of Ledoux's expansion coefficient \cite[Proposition~1.13]{RefLedoux2001}, Ledoux has re-stated the result by Gromov and Milman in terms of the concentration function:
\begin{thm}[e.g., p. 364 of \cite{RefLedoux1990}]
\label{thm:Gromov-Milman}
$M$ has exponential concentration:
\begin{equation*}
{\alpha}_{(M, d_M, \mu_M )} (r) \le C_1 \exp (- C_2 r) 
\quad \mbox{for all $r \ge 0$},
\end{equation*}
where $C_1 = 3/4$ and $C_2 = \sqrt{\lambda_1 (M)} \ln (3/2)$.
\end{thm}
\begin{thm}[Cf. p.~48 of \cite{RefLedoux2001}]
\label{thm:Expansion coefficient of compact Rimennian manifold}
We have
\begin{equation*}
\Expansion_{\Ledoux} ( M; 1 - \varepsilon, \rho ) \ge 1 + \lambda_1 (M) \varepsilon {\rho}^2 \quad \mbox{for some $\rho > 0$}.
\end{equation*}
\end{thm}
\begin{proof}
We omit the details because the assertion actually follows from a slight change in the proof by Ledoux. 
\end{proof}
We will address Gromov's expansion coefficient of $M$; see Example~\ref{ex:example of Gromov problem}.
\subsection{Exponential concentration in terms of expansion coefficients}
In this subsection, we 
will
show that 
Gromov's and Ledoux's expansion coefficients give rise to exponential concentration.
\subsubsection{Ledoux's expansion coefficient}
The following proposition provides the upper bound for ${{\alpha}^{\varepsilon}}_{(X, d_X, \mu_X )}$ in terms of $\Expansion_{\Ledoux} ( X; \varepsilon, \rho )$ only; cf. Proposition~1.13 of \cite{RefLedoux2001}, in which Ledoux has discussed the case 
where
$\varepsilon = 1/2$ especially, although his result is 
not
in terms of 
the expansion coefficient $\Expansion_{\Ledoux} ( X; 1/2, \rho )$ but $e \, (\le \Expansion_{\Ledoux} ( X; 1/2, \rho ))$.
\begin{thm}
\label{thm:bound for diam by Ledoux}
Set a positive real number $\varepsilon < 1$. For each $r > 0$, 
select
$\rho > 0$ such that $\rho \le r$. Then, we have
\begin{multline}
{{\alpha}^{\varepsilon}}_{(X, d_X, \mu_X )} (r) \le 
( 1 - \varepsilon ) \Expansion_{\Ledoux} ( X; 1 - \varepsilon, \rho ) 
\\
\cdot
( \Expansion_{\Ledoux} ( X; 1 - \varepsilon, \rho ) )^{- r /\rho}.
\label{eq:upper bound with Ledoux only for concentration function}
\end{multline}
In particular, if $\Expansion_{\Ledoux} ( X; 1 - \varepsilon, \rho ) > 1$, then $X$ has exponential concentration as 
follows:
\begin{multline}
{{\alpha}^{\varepsilon}}_{( X, d_X, \mu_X )} (r) \le ( 1 - \varepsilon ) \Expansion_{\Ledoux} ( X; 1 - \varepsilon, \rho ) 
\\
\cdot
\exp \left( - (\ln \Expansion_{\Ledoux} ( X; 1 - \varepsilon, \rho )) r / \rho \right).
\label{eq:exponential concentration with Ledoux only concerning mainresult}
\end{multline}
\end{thm}

\begin{proof}
We first interpolate each $r > 0$ between $k \rho$ and $(k + 1) \rho$ for some $k \in \N$. 
Let $A \subset X$ with $\mu_X (A) \ge \varepsilon$. Put $B := {A_{k \rho}}^c$. We see 
that
$B_{k \rho} \backslash B \subset A_{k \rho} \backslash A$, namely $B_{k \rho} \subset A^c$. Hence,
\begin{equation}
\mu_{X} (B_{k \rho}) \le 1 - \mu_X (A) \le 1 - \varepsilon.
\label{eq:2_For expansion coefficient by Ledoux property}
\end{equation}
Furthermore, 
combining
\eqref{eq:2_For expansion coefficient by Ledoux property} with \eqref{eq:iterative Ledoux Expansion coefficient} yields
\begin{align}
\begin{split}
\mu_X (B_{k \rho})
& \ge (\Expansion_{\Ledoux} ( X; 1 - \varepsilon, \rho ))^k \mu_X (B) 
\\
& = (\Expansion_{\Ledoux} ( X; 1 - \varepsilon, \rho ))^k ( 1 - \mu_{X} (A_{ k \rho}) ).
\end{split}
\label{eq:lower bound for B_{k rho}}
\end{align} 
Hence, by adding \eqref{eq:2_For expansion coefficient by Ledoux property} and \eqref{eq:lower bound for B_{k rho}}, we obtain
\begin{equation}
\mu_X (A_{k \rho}) \ge \frac{{\Expansion_{\Ledoux} ( X; 1 - \varepsilon, \rho )}^k - ( 1 - \varepsilon )}{{\Expansion_{\Ledoux} ( X; 1 - \varepsilon, \rho )}^k}.
\label{eq:lower bound of measure of A in Ledoux}
\end{equation}
Consequently, we deduce from the interpolation of $k$ and \eqref{eq:lower bound of measure of A in Ledoux} that
\begin{align*}
1 - \mu_X (A_r) 
&\le
1 - \mu_X (A_{k \rho}) 
\notag
\\
& \le \frac{1 -  \varepsilon}{( \Expansion_{\Ledoux} ( X; 1 - \varepsilon, \rho ) )^k}
\notag
\\
& \le \frac{1 - \varepsilon}{( \Expansion_{\Ledoux} ( X; 1 - \varepsilon, \rho ) )^{\frac{r}{\rho} - 1}}.
\end{align*}
\begin{sloppypar}
Therefore, we obtain \eqref{eq:upper bound with Ledoux only for concentration function}, from which \eqref{eq:exponential concentration with Ledoux only concerning mainresult} follows readily whenever $\Expansion_{\Ledoux} (X; 1 - \varepsilon, \rho) > 1$. 
This completes the proof.
\end{sloppypar}
\end{proof}
\subsubsection{Gromov's and Ledoux's expansion coefficients: The key to Gromov's problem}
The following theorem plays a key role in Gromov's problem and a by-product; see Corollary~\ref{cor:upper bound for expansion coefficient by Gromov}. 
\begin{thm}
\label{thm:evaluation of concentration function using Gromov expansion coefficient}
Set a positive real number $\varepsilon < 1$. For each $r > 0$, select $\rho > 0$ such that $\rho \le r$. 
Then, we have
\begin{multline*}
{{\alpha}^{\varepsilon}}_{(X, d_X, \mu_X )} (r) \le  
( 1 - \varepsilon ) \Expansion_{\Gromov} ( X; \varepsilon, \rho ) (\Expansion_{\Ledoux} ( X; 1 - \varepsilon, \rho ) )^2 
%
\\
\cdot
( \Expansion_{\Ledoux} ( X; 1 - \varepsilon, \rho ) )^{- r / \rho}.
\end{multline*}
In particular, if $\Expansion_{\Ledoux} ( X; 1 - \varepsilon, \rho ) > 1$, then 
$X$
has exponential concentration as 
follows:
\begin{multline}
{{\alpha}^{\varepsilon}}_{(X, d_X, \mu_X )} (r) \le 
( 1 - \varepsilon ) \Expansion_{\Gromov} ( X; \varepsilon, \rho ) (\Expansion_{\Ledoux} ( X; 1 - \varepsilon, \rho ) )^2 
\\
\cdot
\exp ( - ( \ln \Expansion_{\Ledoux} ( X; 1 - \varepsilon, \rho )) r / \rho ).
\label{eq:exponential concentration concerning mainresult}
\end{multline}
\end{thm}

\begin{proof}
The strategy of the current proof is 
similar to that
of the proof of Theorem~\ref{thm:bound for diam by Ledoux}. Interpolate each $r > 0$ between $k \rho$ and $(k + 1) \rho$ for some $k \in \N$. 
Let $A \subset X$ with $\mu_X (A) \ge \varepsilon$. Put $B := {A_{( k - 1 ) \rho}}^c$, where $A_0 := A$ if $k = 1$. 
Likewise,
\begin{equation*}
\mu_{X} (B_{( k - 1 ) \rho}) \le 1 - \mu_{X} (A) \le 1 - \varepsilon.
\end{equation*}
Under the same reasoning as that of 
the proof of Theorem~\ref{thm:bound for diam by Ledoux}, we have
\begin{align*}
\mu_X (B_{( k - 1 ) \rho})
& \ge (\Expansion_{\Ledoux} ( X; 1 - \varepsilon, \rho ))^{k - 1} \mu_X (B) \quad \text{}
\label{eq:inductive evaluate}
\\
& = (\Expansion_{\Ledoux} ( X; 1 - \varepsilon, \rho ))^{k - 1} ( 1 - \mu_{X} (A_{( k - 1 ) \rho}) ),
\end{align*}
where
replacing $A$ with $A_{\rho}$ and by using $A_{\ell \rho} \supset (A_{ ( \ell - 1 ) \rho})_{\rho}$ for each $\ell \in \N$, we 
get
\begin{align*}
1 - ( \Expansion_{\Ledoux} ( X; 1 - \varepsilon, \rho ) )^{k - 1} ( 1 - \mu_{X} (A_{k \rho}) ) 
& \ge
\mu_{X} (A_{\rho})
\\
& \ge \Expansion_{\Gromov} ( X; \varepsilon, \rho ) \varepsilon,
%
\end{align*} 
where we have used \eqref{eq:property of expansion coefficient by Gromov} in the last inequality.
Hence, we obtain
\begin{equation}
\mu_{X} (A_{k \rho}) 
\ge
\frac{( \Expansion_{\Ledoux} ( X; 1 - \varepsilon, \rho ) )^{k - 1} - ( 1 - \Expansion_{\Gromov} ( X; \varepsilon, \rho ) \varepsilon )}{( \Expansion_{\Ledoux} ( X; 1 - \varepsilon, \rho ) )^{k - 1}} > 0,
\label{eq:lower bound of measure of A in Gromov-Ledoux}
\end{equation}
where the last inequality is due to \eqref{eq:property of expansion coefficient by Gromov} and Definition~\ref{def:Expansion coefficient by Ledoux}.

Hence, we conclude from the interpolation of $k$ and \eqref{eq:lower bound of measure of A in Gromov-Ledoux} that
\begin{align}
1 - \mu_X (A_r) 
&\le
1 - \mu_X (A_{k \rho}) 
\notag
\\
& \le \frac{1 - \Expansion_{\Gromov} ( X; \varepsilon, \rho ) \varepsilon}{( \Expansion_{\Ledoux} ( X; 1 - \varepsilon, \rho ) )^{ k - 1}}
\notag
\\
& \le \frac{1 - \Expansion_{\Gromov} ( X; \varepsilon, \rho ) \varepsilon}{( \Expansion_{\Ledoux} ( X; 1 - \varepsilon, \rho ) )^{\frac{r}{\rho} - 2}}
\label{eq:for the Corollary to Gromov}
\\
& \le \frac{( 1 - \varepsilon ) \Expansion_{\Gromov} ( X; \varepsilon, \rho )}{( \Expansion_{\Ledoux} ( X; 1 - \varepsilon, \rho ) )^{\frac{r}{\rho} - 2}}.
\notag
\end{align}
Following
the proof of the conclusion of Theorem~\ref{thm:bound for diam by Ledoux}, 
we obtain the desired result.
\end{proof}
{

\begin{rem}
\label{rem:upper bounds for the concentration fct}
One sees immediately that the upper bounds for ${{\alpha}^{\varepsilon}}_{( X, d_X, \mu_X )}$ in Theorem~\ref{thm:bound for diam by Ledoux} are more sharper than 
those
in Theorem~\ref{thm:evaluation of concentration function using Gromov expansion coefficient}.  
\end{rem}
\begin{rem}
\label{rem:byproduct}
The procedure for the current proof gives more, namely \eqref{eq:for the Corollary to Gromov} permits us to obtain the upper bound for the expansion coefficient in terms of the observable diameter; see Corollary~\ref{cor:upper bound for expansion coefficient by Gromov}.
\end{rem}

Now, as we have assumed that $\Expansion_{\Ledoux} ( X; 1 - \varepsilon, \rho ) > 1$ in Theorems \ref{thm:bound for diam by Ledoux} and \ref{thm:evaluation of concentration function using Gromov expansion coefficient}, 
it is reasonable to ask 
when $\Expansion_{\Ledoux} ( X; 1 - \varepsilon, \rho ) > 1$. Thus, in Appendix~\ref{appendix:Appendix A}, we 
will
address a sufficient condition for the assumption.
%
\section{Observation for Gromov's problem, and by-products}
\label{sec:The observation for Gromov's problem, and by-products}
In this section, we will state the main result on Gromov's problem, which gives the answer to Gromov's problem; see Exercise~\ref{exercise:Gromov} 
given
in Section~\ref{sec:Introduction}. Furthermore, the proof of the procedure for the answer enables one to observe the
upper
bounds for Gromov's expansion coefficient and the observable diameter in terms of Ledoux's expansion coefficient.
%
\subsection{Lower bounds for the 
expansion coefficients: An answer to Gromov's problem and its application to a Riamannian manifold}
The following theorem is our answer to Gromov's problem:
\begin{thm}
\label{thm:answer to Gromov}
Let $\varepsilon \le 1/2$.
Assume that
$\Expansion_{\Ledoux} ( X; 1 - \varepsilon, \rho ) > 1$ for some $\rho > 0$.
Then, 
$\Expansion_{\Gromov} ( X; \varepsilon, \rho )$  
is bounded from below in terms of $\ObsDiam ( X; - \kappa )$,
$0 < \kappa <1$, and $\Expansion_{\Ledoux} ( X; 1 - \varepsilon, \rho )$ 
as follows:
\begin{equation}
\Expansion_{\Gromov} ( X; \varepsilon, \rho ) 
\ge
\frac{\kappa \exp \left( \ObsDiam ( X; - \kappa ) \ln ( \Expansion_{\Ledoux} ( X; 1 - \varepsilon, \rho ) ) / 2 \rho \right)}{2 (1 - \varepsilon) (\Expansion_{\Ledoux} ( X; 1 - \varepsilon, \rho ))^2}.
\label{eq:answer to Gromov}
\end{equation}
\end{thm}

\begin{proof}
Let $\rho > 0$ satisfy the hypotheses of Theorem~\ref{thm:evaluation of concentration function using Gromov expansion coefficient}.
To establish the current theorem, we first apply 
\eqref{eq:sequential Concentration-ObservableDiameter}
to \eqref{eq:exponential concentration concerning mainresult}; accordingly,
\begin{multline*}
\ObsDiam ( X; - \kappa ) 
\\
\le \frac{2 \rho \ln \left( 2 (1 - \varepsilon) \Expansion_{\Gromov} ( X; \varepsilon, \rho ) {( \Expansion_{\Ledoux} ( X; 1 - \varepsilon, \rho ) )}^2 / \kappa \right)}{\ln \Expansion_{\Ledoux} ( X; 1 - \varepsilon, \rho )},
\end{multline*}
from which the desired result follows computationally.
%
\end{proof}

\begin{example}
\label{ex:example of Gromov problem}
Let us now return to the context of Subsection~\ref{subsec:Application to a compact Riemannian manifold}. 
Under the hypotheses of Theorem~\ref{thm:answer to Gromov}, we 
will
be concerned with $M$ having 
$\Ric (M) \ge 0$. 

Now, the following claim permits us to estimate Ledoux's expansion coefficient in terms of 
$\lambda_1(M)$
and the doubling constant of the Riemannian measure of $M$:
\begin{claim}
\label{claim:Ledoux's expansion coefficient}
Let $A$ and $B$ be subsets of $M$ such that 
$d_M (A, B) > 0$
and $\mu_M (A) \ge \varepsilon$.
Now, we deduce that for each $\rho > 0$
\begin{equation*}
\mu_M (B(x, 2 \rho)) \le \mu_M (B_{\rho}) \le 1 - \varepsilon \quad \mbox{for some $x \in M$}.
\end{equation*}
Here, 
taking
$\rho > 0$ such that $\mu_M (B_{\rho}) = 1 - \mu_M (A)$, 
we have
\begin{equation}
1 + \lambda_1 (M) \varepsilon {\rho}^2 \le \Expansion_{\Ledoux} ( M; 1 - \varepsilon, \rho ) \le 2^n,
\label{eq:bounds for Ledoux's expansion coefficient}
\end{equation}
where the first and 
second inequalities are 
given
by Theorem~\ref{thm:Expansion coefficient of compact Rimennian manifold} and because the doubling constant of the Riemannian measure of $M$, 
which
is the upper bound for $\Expansion_{\Ledoux} ( M; 1 - \varepsilon, \rho )$, is equal to $2^n$.
The doubling constant is 
given
by the Bishop-Gromov volume comparison theorem
(also called Riemannian volume comparison theorem);
see e.g., \cite{RefBaudoin:Garofalo2011} and pp.~377--378 of \cite{RefVillani2009}. 
\end{claim}
For such a $\rho > 0$, combining \eqref{eq:bounds for Ledoux's expansion coefficient} with \eqref{eq:answer to Gromov}, we have
\begin{equation}
\Expansion_{\Gromov} ( M; \varepsilon, \rho ) 
\ge
\frac{\kappa \exp \left( \ObsDiam ( M; - \kappa ) \ln ( 1 + \lambda_1 (M) \varepsilon {\rho}^2 ) / 2 \rho \right)}{2^{2 n+1} (1 - \varepsilon)}. 
\label{eq:Ledoux's expansion coefficient of Riemannian manifold}
\end{equation}
\end{example}
\begin{rem}
Let $M$ be a compact Riemannian manifold.
\cite{RefShioya2016} has shown that
\begin{equation*}
\Expansion_{\Gromov} ( M; \varepsilon, \rho ) \ge \min \{ 1 + \lambda_1 (M) {\rho}^2 / 4, 2 \} 
\quad \mbox{for all $\rho > 0$},
\end{equation*}
provided $0 < \varepsilon \le 1/4$.
\end{rem}
The current subsection will end up with showing the relation between Ledoux's expansion coefficient and the observable diameter:
\begin{cor}
\label{cor:lower bound for Ledoux's expansion coefficient}
Let $M$ be an $n$-dimensional compact connected Riemannian manifold.
Under the hypotheses of Theorem~\ref{thm:answer to Gromov}, we have, for some $\rho > 0$ that appeared in the context of Example~\ref{ex:example of Gromov problem},
\begin{equation*}
\Expansion_{\Ledoux} ( M; 1 - \varepsilon, \rho ) \ge \frac{\kappa \exp \left( \ObsDiam ( M; - \kappa ) \ln (1 + \lambda_1 (M) \varepsilon {\rho}^2) / 2 \rho \right)}{2 (1 - \varepsilon)}.
\end{equation*}
\end{cor}
%
\begin{proof}
The same reasoning as 
that utilized in
\eqref{eq:Ledoux's expansion coefficient of Riemannian manifold} 
applies to \eqref{eq:exponential concentration with Ledoux only concerning mainresult}, 
hence
the corollary.
\end{proof}
\subsection{By-products: Upper bounds for Gromov's expansion coefficient and for the observable diameter}
\label{sec:By-product}
As mentioned in Remark~\ref{rem:byproduct}, we 
will
show the upper bound for Gromov's expansion coefficient in terms of the observable diameter and Ledoux's expansion coefficient; accordingly, Gromov's expansion coefficient is bounded from above and below by the two geometric quantities. Ultimately, one can derive the upper bound for the observable diameter in terms of Ledoux's expansion coefficient.

The procedure for the proof of Theorem~\ref{thm:evaluation of concentration function using Gromov expansion coefficient} implies the upper bound for Ledoux's expansion coefficient: 
\begin{cor}
\label{cor:upper bound for expansion coefficient by Gromov}
Under the hypotheses of Theorem~\ref{thm:answer to Gromov},
assume further that
\begin{equation}
2 ( 1 - \Expansion_{\Gromov} ( X; \varepsilon, \rho ) \varepsilon ) \Expansion_{\Ledoux} ( X; 1 - \varepsilon, \rho ) \ge \kappa \quad \mbox{for some $\rho > 0$}.
\label{eq:assumption on upper bound for Gromov expansion coefficient}
\end{equation}
Then, the upper bound for $\Expansion_{\Gromov} ( X; \varepsilon, \rho )$  
is given by
\begin{multline}
\Expansion_{\Gromov} ( X; \varepsilon, \rho ) 
\\
\le \frac{2 - \kappa \exp ( ( \ObsDiam ( X; - \kappa ) - 4 \rho ) \ln ( \Expansion_{\Ledoux} ( X; 1 - \varepsilon, \rho )  ) / 2 \rho )}{2 \varepsilon}.
\label{eq:upper bound the expansion coefficient}
\end{multline}
\end{cor}

\begin{proof}
For each $r > 0$, select $\rho > 0$ such that $\rho \le r$. Then, one sees immediately from \eqref{eq:for the Corollary to Gromov} that
\begin{multline}
{{\alpha}^{\varepsilon}}_{(X, d_X, \mu_X )} (r) 
\le ( 1 - \Expansion_{\Gromov} ( X; \varepsilon, \rho ) \varepsilon ) (\Expansion_{\Ledoux} ( X; 1 - \varepsilon, \rho ) )^2 
\\
\cdot
\exp ( - ( \ln \Expansion_{\Ledoux} ( X; 1 - \varepsilon, \rho )) r / \rho ).
\label{eq:using the concentration}
\end{multline}
On account of \eqref{eq:assumption on upper bound for Gromov expansion coefficient}, one can see that the procedure for the proof of Theorem~\ref{thm:answer to Gromov} works for \eqref{eq:using the concentration} as well. We leave it to the reader to verify its computation.
Consequently,
\begin{multline*}
\ObsDiam ( X; - \kappa ) 
\\
\le \frac{2 \rho \ln \left( 2 (1 - \Expansion_{\Gromov} ( X; \varepsilon, \rho ) \varepsilon) {(\Expansion_{\Ledoux} ( X; 1 - \varepsilon, \rho ))}^2 {\kappa}^{-1} \right)}{\ln \Expansion_{\Ledoux} ( X; 1 - \varepsilon, \rho )},
\end{multline*}
whence the current corollary establishes computationally.
\end{proof}

Theorem~\ref{thm:answer to Gromov} and Corollary~\ref{cor:upper bound for expansion coefficient by Gromov} yield the upper bound for the observable diameter:
\begin{cor}
\label{cor:upper bound by Ledoux's expansion coefficient for the observable diameter}
Under the hypotheses of Theorem~\ref{thm:answer to Gromov} and Corollary~\ref{cor:upper bound for expansion coefficient by Gromov}, the upper bound for the observable diameter is in terms of Ledoux's expansion coefficient:
\begin{multline*}
\ObsDiam ( X; - \kappa ) 
\\
\le \frac{2 \rho}{\ln \Expansion_{\Ledoux} ( X; 1 - \varepsilon, \rho )} \ln \left( \frac{2 {(\Expansion_{\Ledoux} ( X; 1 - \varepsilon, \rho ))}^2 (1 - \varepsilon)}{( {1 + (\Expansion_{\Ledoux} ( X; 1 - \varepsilon, \rho ))}^2 ) \varepsilon \kappa} \right).
\end{multline*}
\end{cor}
\begin{proof}
By 
combining
\eqref{eq:answer to Gromov} and
\eqref{eq:upper bound the expansion coefficient},
we have
\begin{multline*}
\frac{\kappa \exp \left( \ObsDiam ( X; - \kappa ) \ln ( \Expansion_{\Ledoux} ( X; 1 - \varepsilon, \rho ) ) / 2 \rho \right)}{2 (1 - \varepsilon) (\Expansion_{\Ledoux} ( X; 1 - \varepsilon, \rho ))^2} 
\\
\le \frac{2 - \kappa  \exp ( ( \ObsDiam ( X; - \kappa ) - 4 \rho ) \ln ( \Expansion_{\Ledoux} ( X; 1 - \varepsilon, \rho )  ) / 2 \rho ) }{2 \varepsilon},
\end{multline*}
from which the desired result follows computationally.
\end{proof}
%
%
\begin{example}
\label{ex:result on the upper bounds for observable diameter}
Let $M$ be an $n$-dimensional compact connected Riemannian manifold with $\Ric (M) \ge 0$. 
Corollary~\ref{cor:upper bound by Ledoux's expansion coefficient for the observable diameter} yields the upper bound for the observable diameter of $M$ in terms of the doubling constant of the Riemannian measure of $M$, which is equal to $2^n$, and $\lambda_1 (M)$. 
We have for some $\rho > 0$, which appeared in the context of Example~\ref{ex:example of Gromov problem},
\begin{multline*}
\ObsDiam ( M; - \kappa ) 
\\
\le \frac{2 \rho}{\ln (1 + \lambda_1 (M) \varepsilon {\rho}^2)} \min \biggl\{ \ln \frac{2^{2n+1} (1 - \varepsilon)}{( 1 + {(1 + \lambda_1 (M) \varepsilon {\rho}^2)}^2 ) \varepsilon \kappa}, \ln \frac{2 (1 - \varepsilon)}{\varepsilon \kappa} \biggr\};
\end{multline*}
cf. Example~\ref{ex:Observable diameter of Riemannian manifold}. 
\end{example}

\section{Estimate for the diameter of a bounded metric measure space}
The intent of the present appendix is to give an upper bound and a lower bound for the diameter of certain metric measure spaces. 
\subsection{Upper bound for the diameter}
\label{subsection:upper bound for the diameter}
The exponential concentration obtained in 
\eqref{eq:exponential concentration with Ledoux only concerning mainresult} and Theorem~\ref{thm:characterization of the doubling measure} are clues to the proof of Theorem~\ref{thm:bound for diam by Gromov and Ledoux}.
\begin{thm}
\label{thm:bound for diam by Gromov and Ledoux}
Let $(X, d_X, \mu_X)$ be a bounded metric measure space with a doubling measure, whose doubling constant is denoted by $C$, and satisfy the requirement $\min\{ \Expansion_{\Ledoux} ( X; \varepsilon, \rho ), \Expansion_{\Ledoux} ( X; 1 - \varepsilon, \rho ) \} > 1$ with $\varepsilon \le 1/2$.
Then, its diameter is bounded from above in terms of Ledoux's expansion coefficient and $C$ 
as follows:
\begin{multline*}
\diam (X) 
\le 3 \rho \max
\biggl\{\frac{\ln \left( C^4 ( 1 - \varepsilon ) {\varepsilon}^{-1} \Expansion_{\Ledoux} ( X; 1 - \varepsilon, \rho ) \right)}{\ln \Expansion_{\Ledoux} ( X; 1 - \varepsilon, \rho )},
\\
\frac{2 \ln \left( C^3 3^{\ln C /\! \ln 2} \varepsilon \Expansion_{\Ledoux} ( X; \varepsilon, \rho ) \right)}{\ln \Expansion_{\Ledoux} ( X; \varepsilon, \rho )} \biggr\}, \quad \rho > 0.
\end{multline*}
\label{thm:Estimate for the diameter of mm-sp}
\end{thm} 

\begin{proof}
As shown in Appendix~\ref{appendix:Appendix A}, 
the assumption on the doubling measure implies that Ledoux's expansion coefficient is finite; see \eqref{eq:sufficient condition for finiteness of Ledoux}. 
Theorem~\ref{thm:evaluation of concentration function using Gromov expansion coefficient} 
is a crux to prove this theorem. It follows from 
\eqref{eq:exponential concentration with Ledoux only concerning mainresult} that for all $r \ge 0$ and for all $A \subset X$, such that $\mu_X (A) \ge \varepsilon$,
\begin{equation}
1 - \mu_X (A_r) 
\\
\le 
( 1 - \varepsilon ) \Expansion_{\Ledoux} ( X; 1 - \varepsilon, \rho ) \exp ( - ( \ln \Expansion_{\Ledoux} ( X; 1 - \varepsilon, \rho ) ) r /\! \rho ).
\label{eq:exponential concentration concerning mainresult 2}
\end{equation}
Let us regard $r$ as being sufficiently small and fixed. 
To estimate $\diam (X)$, we consider a ball with radius $\tau \diam (X)$ 
centred
at $x \in X$ attaining $\diam (X)$, where $\tau \le 1$ is a positive parameter of $\diam (X)$. 
We need to observe 
whether
$\tau$ is reasonable. Then, for the desired upper bound to be sharp, 
note that such a point $x$ makes it allowable, for which we refer the reader to Remark~\ref{rem:reasonable parameter}. 
Take now a distinct point $z \in X$ of $x$ such that 
\begin{equation*}
d_X (x, z) = 2r + \tau \diam (X) \, (\, \le \diam (X))
\end{equation*}
for some $\tau$ such that $\tau \diam (X) \le r$. Hence, 
\begin{equation}
\tau \le 1/3.
\label{eq:parameter bound} 
\end{equation}

Hereafter,
we will evaluate the measure of the ball $B (x, \tau \diam (X))$ separately by means of $\varepsilon$. 
It follows 
that
\begin{gather}
\mu_X ( B (z,r) ) \le 1 - \mu_X (B (x, r + \tau \diam (X))) \le 1 - \mu_X ({B (x, \tau \diam (X))}_r)
\label{eq:mu_X ( B (z,r) ) < 1 - mu_X ({B (x, tau diam (X))}_r)},
\\
B (x, \tau \diam (X)) \subset B (z, 2 ( r + \tau \diam (X) )).
\label{eq:B_{{t} diam (X)} (x) subset B_{2 ( {t} diam (X) + r )} (z)}
\end{gather}
We now apply Theorem~\ref{thm:characterization of the doubling measure} to $B(z, 2 ( r + \tau \diam (X) ))$ and $B (z,r)$ to obtain
\begin{equation*}
\frac{\mu_X (B (z, 2 ( r + \tau \diam (X) )))}{\mu_X ( B (z,r) )} \le C^2 \left( \frac{2 ( r + \tau \diam (X) )}{r} \right)^{\ln C / \ln 2},
\end{equation*}
from which it follows that 
\begin{equation}
\mu_X ( B (z,r) ) 
\ge C^{-2} {\left( \frac{2 ( r + \tau \diam (X) )}{r} \right)}^{- \ln C /\! \ln 2} \mu_X (B( z, 2 (r + \tau \diam (X) ) )).
\label{lower bound for the measure of the ball centred z}
\end{equation}
Combining \eqref{lower bound for the measure of the ball centred z} with \eqref{eq:B_{{t} diam (X)} (x) subset B_{2 ( {t} diam (X) + r )} (z)}, 
we have
\begin{equation}
\mu_X ( B (z,r) ) 
\ge C^{-2} {\left( \frac{2 ( r + \tau \diam (X) )}{r} \right)}^{- \ln C /\! \ln 2} \mu_X (B (x, \tau \diam (X))).
\label{eq:lower bound for ball}
\end{equation}

We begin by letting $A = B (x, \tau \diam (X))$ with $\mu_X (B (x, \tau \diam (X)) \ge \varepsilon$.
Then, it follows from \eqref{eq:mu_X ( B (z,r) ) < 1 - mu_X ({B (x, tau diam (X))}_r)} and \eqref{eq:lower bound for ball} that
\begin{equation}
C^{-2} {\left( \frac{2 ( r + \tau \diam(X) )}{r} \right)}^{- \ln C /\! \ln 2} \varepsilon \le 1 - \mu_X ({B (x, \tau \diam (X))}_r) = 1 - \mu_X (A_r).
\label{eq:lower bound for complement of Ar}
\end{equation}
By virtue of \eqref{eq:exponential concentration concerning mainresult 2}, 
plugging it with \eqref{eq:lower bound for complement of Ar}, we obtain
\begin{multline}
C^{-2} {\left( \frac{2 ( r + \tau \diam(X) )}{r} \right)}^{- \ln C /\! \ln 2} \varepsilon 
\\
\le 
( 1 - \varepsilon ) \Expansion_{\Ledoux} ( X; 1 - \varepsilon, \rho ) \exp ( - ( \ln \Expansion_{\Ledoux} ( X; 1 - \varepsilon, \rho ) ) r /\! \rho ).
\label{eq:adaptation of exponential concentration concerning mainresult 2}
\end{multline}
Now, $r$ being arbitrary, especially letting $r = \tau \diam (X)$ in \eqref{eq:adaptation of exponential concentration concerning mainresult 2}, gives
\begin{equation}
\diam (X) \le \frac{\rho \ln \left( C^4 ( 1 - \varepsilon ) {\varepsilon}^{-1} \Expansion_{\Ledoux} ( X; 1 - \varepsilon, \rho ) \right)}{\tau \ln \Expansion_{\Ledoux} ( X; 1 - \varepsilon, \rho )}
\label{eq:estimate for diameter1}
\end{equation}
whenever $\varepsilon \le 1/2$ for the upper bound of \eqref{eq:estimate for diameter1} to be legitimate. 

We next 
consider
the case $\mu_X ( B (x, \tau \diam (X) ) < \varepsilon$. Then, letting $A$ be the complement of $B (x, \tau \diam (X))$
furnishes that 
\begin{gather}
\mu_X (A) \ge 1 - \varepsilon,
\label{eq:mu_X (A) ge varepsilon}
\\
B (x, \tau \diam (X) / 2) \subset ( A_{\tau \diam (X) / 2} )^c
\label{eq:measure of ball and complement}
\end{gather}
because 
\begin{equation*}
A_{\tau \diam (X) / 2} \subset (B (x, \tau \diam (X) - \tau \diam (X) / 2))^c = (B (x, \tau \diam (X) / 2))^c.
\end{equation*}
Similarly to
the case 
where
$\mu_X (B (x, \tau \diam (X))) \ge \varepsilon$, Theorem~\ref{thm:characterization of the doubling measure} allows us to deduce that
\begin{align*}
\frac{\mu_X (B (x,\diam (X)))}{\mu_X ( B (x, \tau \diam (X) / 2) )} 
&= \frac{1}{\mu_X ( B (x, \tau \diam (X) / 2) )}
\\
&\le C^2 \left( \frac{\diam (X)}{ \tau \diam (X) / 2} \right)^{\ln C /\! \ln 2}
\\
&= C^2 \left( \frac{2}{\tau} \right)^{\ln C /\! \ln 2},
\end{align*}
which indicates that
\begin{equation}
C^{-2} \left( \frac{2}{\tau} \right)^{-\ln C /\! \ln 2} 
\le \mu_X ( B (x, \tau \diam (X) / 2) ).
\label{eq:measure of ball centred x}
\end{equation}
Combining \eqref{eq:measure of ball centred x} with \eqref{eq:measure of ball and complement}, we have
\begin{equation*}
C^{-2} \left( \frac{2}{\tau} \right)^{-\ln C /\! \ln 2} 
\le 1 - \mu_X ( A_{\tau \diam (X) / 2 } ). 
\end{equation*}
Furthermore, under \eqref{eq:mu_X (A) ge varepsilon}, \eqref{eq:exponential concentration concerning mainresult 2} leads to
\begin{multline}
C^{-2} \left( \frac{2}{\tau} \right)^{-\ln C /\! \ln 2} 
\\
\le \varepsilon \Expansion_{\Ledoux} ( X; \varepsilon, \rho ) \exp ( - (\ln \Expansion_{\Ledoux} ( X; \varepsilon, \rho )) \tau \diam (X) / 2 \rho ). 
\label{eq:for diameter}
\end{multline}
It follows from \eqref{eq:for diameter} that
\begin{equation}
\diam (X) \le \frac{2 \rho \ln \left( C^3 {\tau}^{- \ln C /\! \ln 2} \varepsilon \Expansion_{\Ledoux} ( X; \varepsilon, \rho ) \right)}{\tau \ln \Expansion_{\Ledoux} ( X; \varepsilon, \rho )}.
\label{eq:estimate for diameter2}
\end{equation}
Finally, 
we 
note
that \eqref{eq:parameter bound} gives more, namely one can eventually conclude from \eqref{eq:estimate for diameter2} that letting $\tau = 1/3$ attains the upper bound for the diameter 
appropriately,
hence
the theorem.
\end{proof}

\begin{rem}
\label{rem:reasonable parameter}
One sees immediately that letting the parameter $\tau$ be maximal makes it allowable that the upper bound for $\diam (X)$ of \eqref{eq:estimate for diameter2} is sharp. In fact, distinct points attaining $\diam (X)$ make $\tau$ be so.
\end{rem}
\begin{rem}
To make \eqref{eq:estimate for diameter1} 
legitimate, 
with
an assumption on $\varepsilon$, it is sufficient to assume that $\varepsilon \le 1/2$ for \eqref{eq:estimate for diameter1} only. In fact, if \eqref{eq:estimate for diameter2} 
does not appear correct,
the desired upper bound will be taken as \eqref{eq:estimate for diameter1}. 
\end{rem}

\begin{rem}
In order to get the sharper upper bound, it is adequate in the proof of Theorem~\ref{thm:bound for diam by Gromov and Ledoux} to adopt not \eqref{eq:exponential concentration concerning mainresult} but \eqref{eq:exponential concentration with Ledoux only concerning mainresult}. 
\end{rem}
\begin{rem}
To 
the
best of our knowledge, Naor et al. were the first 
ones 
to 
show
that the upper bound for the diameter of a certain bounded metric measure space with doubling constant $C \, (> 3)$ 
is in terms of the observable diameter; see Theorem~1.7 of \cite{RefNaor2005} for more rigorous treatments.
\end{rem}
The current subsection will end up with discussing the upper bounds for the diameters of some metric measure spaces and applying Theorem~\ref{thm:bound for diam by Gromov and Ledoux} to a Riemannian manifold. The following is the most 
well-known
diameter estimate and control theorem for a Riemannian manifold, which goes back as far as Myers \cite{RefMyers1935} and \cite{RefMyers1941}, and is currently called `(Bonnet-)Myers theorem'; see e.g., Section~1 of \cite{RefBakry:Ledoux1996} and p.~378 of \cite{RefVillani2009} for a modern treatment:
\begin{thm}[(Bonnet-)Myers theorem]
Let $M$ be an $n \, (\ge 2)$-dimensional complete connected Riemannian manifold with $\Ric (M) \ge K > 0$. Then, we have
\begin{equation}
\diam (M) \le \pi \sqrt{\frac{n-1}{K}}.
\label{eq:Bonnet-Myers}
\end{equation}
Furthermore, $M$ is compact. 

Thereafter, Cheng has shown that the equality of \eqref{eq:Bonnet-Myers} holds if and only if $M$ is isometric to an $n$-dimensional Euclidean sphere ${\Sphere}^n (\delta)$ of radius $\delta > 0$ with constant sectional curvature $K$ given by $(n-1) / {\delta}^2$, which is referred to as the `generalized Toponogov sphere theorem'; see Theorem~3.1 of \cite{RefCheng1975}. 
\end{thm}
The (Bonnet-)Myers theorem for a metric measure space has been established by J.~Lott and C.~Villani, K.-T.~Sturm, and S.~Ohta; see \cite{RefLott:Villani2007} and \cite{RefLott:Villani2009}, \cite{RefSturm2006_1} and \cite{RefSturm2006_2}, and \cite{RefOhta2007} for detailed accounts.
Combining
Theorem~\ref{thm:bound for diam by Gromov and Ledoux} 
with
Claim~\ref{claim:Ledoux's expansion coefficient} 
provides
the
upper bound for the diameter of a compact connected Riemannian manifold in terms of the doubling constant of the Riemannian measure of the manifold and the first non-trivial eigenvalue of the Laplacian on the manifold:
\begin{example}
\label{ex:upper bound for the diameter of a Riemannian manifold}
Let $M$ be an $n$-dimensional compact connected Riemannian manifold with $\Ric (M) \ge 0$. Let $\rho > 0$ be in the context of Claim~\ref{claim:Ledoux's expansion coefficient}. Applying Theorem~\ref{thm:bound for diam by Gromov and Ledoux} to $M$, we computationally derive by Claim~\ref{claim:Ledoux's expansion coefficient}: 
\begin{equation*}
\diam (M) \le
3 \rho \max
\biggl\{\frac{\ln \left( 2^{5n} ( 1 - \varepsilon ) {\varepsilon}^{-1} \right)}{\ln \left( 1 + \lambda_1 (M) \varepsilon {\rho}^2 \right)},
\frac{2 \ln \left( 2^{4n} 3^n \varepsilon \right)}{\ln \left( 1 + \lambda_1 (M) (1 - \varepsilon) {\rho}^2 \right)} \biggr\}.
\end{equation*}
\end{example}
As the preceding study on the current subsection, the diameter upper bounds with spectral (see \cite[Section~3]{RefLedoux2004}) or with logarithmic Sobolev constant (see \cite[Section~4]{RefLedoux2004}) are discussed in the refeences therein. 

\subsection{Lower bound for the diameter}
In the present subsection, we 
will
show that the lower bound for the diameter of a bounded metric measure space is in terms of the Laplace functional on metric measure spaces. 
As with the preceding study on the current subsection, the diameter lower bounds with spectral gap (see \cite[Section~3]{RefLedoux2004}) or with a logarithmic Sobolev constant (see \cite[Section~4]{RefLedoux2004}) are discussed in the refeences therein. 
\begin{defn}[Laplace functional; cf. Section~1.6 of \cite{RefLedoux2001} and Section~1 of \cite{RefLedoux2004}]
\label{def:LaplaceFunctional}
\begin{equation*}
{\Lap}_{(X, d_X, \mu_X)} (\lambda) = \sup \int_{X} \exp \left( \lambda f(x) \right)\,d\mu_{X} 
\quad \text{for all $\lambda > 0$},
\end{equation*}
where the supremum runs over all bounded 1-Lipschitz functions with mean zero on $X$. We call ${\Lap}_{(X, d_X, \mu_X)}$ the \textit{Laplace functional} of $\mu_X$ on $X$.   

In \cite[Section~1.6]{RefLedoux2001}, the Laplace functional is defined for $\lambda = 0$ as well, whereas we will be concerned only with $\lambda > 0$. 
\end{defn}
The Laplace functional allows us to establish the concentration measure phenomenon on a bounded Cartesian metric measure space; see \cite[Section~1.6]{RefLedoux2001}.
%
\begin{prop}[Proposition of \cite{RefLedoux1992}]
\label{prop:Ledoux}
Let $M$ be a compact Riemannian manifold with Ricci curvature $\Ric (M)$ bounded below from a positive constant. Under the hypotheses of Definition~\ref{def:LaplaceFunctional}, Ledoux shows that
\begin{equation*}
\int_{M} \exp \left( \lambda f(x) \right)\,d\mu_{M} \le \exp ({\lambda}^2 / 2 \Ric (M)).
\end{equation*}
\end{prop}
It follows immediately from Proposition~\ref{prop:Ledoux} that
\begin{cor}
\begin{equation*}
{\Lap}_{(M, d_M, \mu_M)} (\lambda) \le \exp ({\lambda}^2 / 2 \Ric (M)).
\end{equation*}
\end{cor}
\begin{thm}
\label{thm:Lower bound}
Let $X$ be a bounded metric measure space. 
For each $\lambda > 0$: 
If $f$ is a bounded Lipschitz function with mean zero on $X$, then 
\begin{equation}
\diam(X) \ge \frac{\ln \int_{X} \exp \left( \lambda f(x) \right)\,d\mu_{X}}{\lambda \| f \|_{\text{Lip}}}.
\label{eq:lower bound for the diameter}
\end{equation}
In particular, if $f$ is a bounded 1-Lipschitz function with mean zero on $X$, then 
\begin{equation}
\diam (X) \ge \frac{\sqrt{2 \ln {\Lap}_{(X, d_X, \mu_X)} (\lambda)}}{\lambda}. 
\label{eq:lower bound for the diamete under 1-Lipschitz}
\end{equation} 
Thus under the hypothesis 
same as \eqref{eq:lower bound for the diamete under 1-Lipschitz}, we deduce that
\begin{equation}
\diam (X) \ge \frac{1}{\lambda} \min \biggl\{ \ln {\Lap}_{(X, d_X, \mu_X)} (\lambda), \sqrt{2 \ln {\Lap}_{(X, d_X, \mu_X)} (\lambda)} \biggr\}.
\label{eq:result on the upper bound for the diameter}
\end{equation}
\end{thm}
\begin{proof}
Our proof starts with estimating an exponential integral of a bounded Lipschitz function $f$ with mean zero on $X$. The strategy of the proof adopts that of Ledoux \cite[Proposition~1.16]{RefLedoux2001}.
However, the argument of the proof by Ledoux is not given in detail. Thus we will explain the argument in detail.
\begin{align}
&\int_{X} \exp \left( \lambda f(x) \right)\,d\mu_{X} 
\notag
\\
&= \int_{X} \exp \left( \lambda f(x) \right)\,d\mu_{X} \exp \left( - \lambda \int_{X} f(y)\,d\mu_{X} \right) 
\label{eq:by mean 0}
\\
&\le \int_{X} \exp \left( \lambda f(x) \right)\,d\mu_{X} \int_{X} \exp \left( - \lambda f(y) \right)\,d\mu_{X} 
\quad \text{by Jensen's inequality} 
\notag
\\
&=
\iint_{X \times X} \exp \left( \lambda \left( f(x) - f(y) \right) \right)\,d\mu_{X} d\mu_{X} 
\quad 
\text{by Fubini's inequality},
\label{eq:by Fubini}
\end{align}
where, in \eqref{eq:by mean 0}, we have used the standing assumption that the mean of $f$ is equal to zero,
and hence
\begin{equation}
\int_{X} \exp \left( \lambda f(x) \right)\,d\mu_{X} \le \iint_{X \times X} \exp \left( \lambda \left( f(x) - f(y) \right) \right)\,d\mu_{X} d\mu_{X}.
\label{eq:evaluation by Jensen}
\end{equation}
We see that 
\begin{equation*}
\iint_{X \times X} \exp \left( \lambda \left( f(x) - f(y) \right) \right)\,d\mu_{X} d\mu_{X} \le \exp \left( \lambda \| f \|_{\text{Lip}} \diam (X) \right)
\end{equation*}
because $f$ is Lipschitz on $X$.
Consequently,
\begin{equation}
\int_{X} \exp \left( \lambda f(x) \right)\,d\mu_{X} \le \exp \left( \lambda \| f \|_{\text{Lip}} \diam (X) \right),
\label{eq:ForEstimateByLaplaceFunctional}
\end{equation}
from which \eqref{eq:lower bound for the diameter} follows. In particular, if $f$ is 1-Lipschitz, then by \eqref{eq:ForEstimateByLaplaceFunctional}, we get
\begin{equation*}
{\Lap}_{(X, d_X, \mu_X)} (\lambda) \le \exp \left( \lambda \diam (X) \right).
\end{equation*}
Hence one sees 
that
\begin{equation}
\diam(X) \ge \frac{\ln {\Lap}_{(X, d_X, \mu_X)} (\lambda)}{\lambda}. 
\label{eq:lower bound for the diamete under 1-Lipschitz mean 0}
\end{equation}

Next, we consider the case 
where
$f$ is a bounded 1-Lipschitz with mean zero on $X$. 
The subsequent claim thus
is the following:
\begin{claim}
\label{claim:convexity}
Under the hypotheses of Theorem~\ref{thm:Lower bound},
\begin{equation*}
\exp (\lambda ( f(x) - f(y) )) \le \cosh (\lambda \diam (X)) + \frac{f(x) - f(y)}{\diam (X)} \sinh (\lambda \diam (X)).
\end{equation*}
\end{claim}
\begin{proof}
From convexity of the exponential function, it follows 
for all $\tau \in \R$ 
and 
for all $x \in \R$ 
such that $|x| \le 1$ that
\begin{align*}
\exp (\tau x) 
&\le \frac{1 + x}{2} \exp (\tau) + \frac{1 - x}{2} \exp (-\tau)
\\
&= \frac{1 + x}{2} \left( \Bigl( \frac{\exp (\tau) + \exp (-\tau)}{2} \Bigr) + \Bigl( \frac{\exp (\tau) - \exp (-\tau)}{2} \Bigr) \right) 
\\
&\qquad + \frac{1 - x}{2} \left( \Bigl( \frac{\exp (\tau) + \exp (-\tau)}{2} \Bigr) - \Bigl( \frac{\exp (\tau) - \exp (-\tau)}{2} \Bigr) \right)
\\
&= \frac{1 + x}{2} \left( \cosh (\tau) + \sinh (\tau) \right) + \frac{1 - x}{2} \left( \cosh (\tau) - \sinh (\tau) \right) 
\\
&= \cosh (\tau) + x \sinh (\tau),
\end{align*}
namely
\begin{equation}
\exp (\tau x) \le \cosh (\tau) + x \sinh (\tau).
\label{eq:covexity}
\end{equation}
Since $f$ is 1-Lipschitz, we readily see that
\begin{equation}
\frac{f(x) - f(y)}{\diam (X)} \le \| f \|_{\text{Lip}} \le 1.
\label{eq:1-Lipschitz constant for convex}
\end{equation}
From \eqref{eq:1-Lipschitz constant for convex}, we see that the current argument is in agreement with \eqref{eq:covexity}. Therefore, Claim~\ref{claim:convexity} holds.
\end{proof}
On account of Claim~\ref{claim:convexity}, we can now continue estimating \eqref{eq:by Fubini} as follows:
\begin{align*}
&\iint_{X \times X} \exp \left( \lambda \left( f(x) - f(y) \right) \right)\,d\mu_{X} d\mu_{X}
\\
&\le 
\iint_{X \times X} \cosh \left( \lambda \diam(X) \right)\,d\mu_{X} d\mu_{X} 
\\
&\phantom{=}
+ \frac{\sinh ( \lambda \diam (X))}{\diam (X)} \iint_{X \times X} \left( f(x) - f(y) \right)\,d\mu_{X} d\mu_{X}
\\
&= \iint_{X \times X} \cosh \left( \lambda \diam(X) \right)\,d\mu_{X} d\mu_{X} 
\\
&= \cosh (\lambda \diam(X))
\\
&\le \sum_{i = 0}^{\infty} \frac{(\lambda \diam (X))^{2i}}{2^i i !} 
\\
&= \sum_{i = 0}^{\infty} \frac{((\lambda \diam (X))^2 / 2)^i}{i!}
\\
&= \exp \left((\lambda \diam (X))^2 / 2\right),
\end{align*}
consequently,
\begin{equation}
\iint_{X \times X} \exp \left( \lambda \left( f(x) - f(y) \right) \right)\,d\mu_{X} d\mu_{X} \le \exp \left((\lambda \diam (X))^2 / 2\right).
\label{eq:upper bound for Bobkov:Ledoux}
\end{equation}
For this reason, we conclude by adding \eqref{eq:evaluation by Jensen} and \eqref{eq:upper bound for Bobkov:Ledoux} that 
\begin{equation*}
\int_{X} \exp \left( \lambda f(x) \right)\,d\mu_{X} \le \exp \left((\lambda \diam (X))^2 / 2\right),
\end{equation*}
from which \eqref{eq:lower bound for the diamete under 1-Lipschitz} follows immediately.
By \eqref{eq:lower bound for the diamete under 1-Lipschitz} and \eqref{eq:lower bound for the diamete under 1-Lipschitz mean 0}, one can arrive at \eqref{eq:result on the upper bound for the diameter}. This therefore proves the theorem. 
\end{proof}

\section{Conclusions}
\label{sec:concluding remarks and discussion}
We 
will
conclude the body of the paper 
by mentioning 
the
work in progress.
The overall aim of 
this
advanced study is to evaluate Gromov's and Ledoux's expansion coefficients in terms of distances and measures on a metric measure space. The study has been motivated by the work 
of
\cite{RefBobkov:Ledoux1997} and \cite{RefChung:Grigoryan:Yau1996}, \cite{RefChung:Grigoryan:Yau1997} and \cite{RefFriedman:Tillich2000}, who have actually derived some bounds for the spectrum 
of the Laplacian on a compact connected Riemannian manifold (as a continuous space) and a graph (as a discrete space). 

The above-mentioned concentration inequality 
stated in Propositions~\ref{prop:concentration inequality} and \ref{prop:seaquential concentration inequality} 
and the Laplace functional
stated in Definition~\ref{def:LaplaceFunctional} 
will play a pivotal role in our advanced work, which will be discussed elsewhere.
\appendix
\section{Sufficient condition for Ledoux's expansion coefficient to be 
$\infty > \Expansion_{\Ledoux} > 1$}
\label{appendix:Appendix A}
Up to Subsection~\ref{subsection:upper bound for the diameter} insomuch as
we have concerned ourselves with 
Ledoux's expansion coefficient to be $\infty > \Expansion_{\Ledoux} ( X; 1 - \varepsilon, \rho ) > 1$, in the present section, we shall address the sufficient condition for the quantity to be so. 

Before the discussion, let us set up the doubling measure, which is often assumed in geometry and analysis on metric measure spaces. We will denote by $B(x, r)$ a ball in $X$ with 
centre
$x \in X$ and radius $r > 0$. 
\begin{defn}[Definition~5.2.1 of \cite{RefAmbrosio:Tilli2004}]
Let $\mathcal{B} (X)$ denote a $\sigma$-algebra of all Borel subsets of $X$. A measure $\mu_X: \mathcal{B} (X) \to [0, + \infty]$ is said to be \textit{doubling} if $\mu_X$ is finite on bounded sets and there exits a constant $C_{\mu_X}$ with respect to $\mu_X$ such that
\begin{equation}
\mu_X ( B (x, 2 r) ) \le C_{\mu_X} \mu_X ( B (x, r) ) 
\quad \mbox{for all $x \in X$ and $r > 0$}.
\label{eq:doubling}
\end{equation}
The best constant $C_{\mu_X} (\ge 1)$ in \eqref{eq:doubling} is called a \textit{doubling constant}, which will be briefly written by $C$. 
\end{defn}
It follows from 
iteration of
\eqref{eq:doubling} that for an arbitrary integer $k \ge 0$
\begin{equation*}
\mu_X ( B ( x, 2^k r ) ) \le C^k \mu_X ( B ( x, r ) ).
\end{equation*}

\begin{example}
A typical example of the doubling measure is Lebesgue measure on a Euclidean space.
\end{example}

The following theorem characterizes the doubling measure $\mu_X$ on metric space $( X, d_X )$ by providing a lower bound for the decay of $r \mapsto \mu_X (B (x, r))$ for $\mu_X$. The characterization will play a 
crucial
role in the estimate for the diameter of $(X, d_X, \mu_X)$ with $\infty > \Expansion_{\Ledoux} ( X; 1 - \varepsilon, \rho ) > 1$, see Theorem~\ref{thm:Estimate for the diameter of mm-sp}.
\begin{thm}[Theorem~5.2.2 of \cite{RefAmbrosio:Tilli2004}]
\label{thm:characterization of the doubling measure}
Let a measure $\mu_X: \mathcal{B} (X) \to [0, + \infty]$ be finite on bounded sets. Then, $\mu_X$ is doubling if and only if there exists a constant $C > 0$ such that
\begin{equation*}
\frac{\mu_X ( B (y, r_2) )}{\mu_X ( B (x, r_1) )} \le C^2 \left( \frac{r_2}{r_1} \right)^{\ln C / \ln 2}
\end{equation*}
for each $r_i$, $i = 1, 2$ such that $0 < r_1 \le r_2$ and all $x, y \in X$ such that $x \in B (y, r_2)$.
\end{thm}
%
\begin{prop}
\label{prop:Expansion coefficient and doubling const}
We first observe the assumption 
that
$\Expansion_{\Ledoux} ( X; 1 - \varepsilon, \rho ) < \infty$.
We will verify that the doubling measure makes $\Expansion_{\Ledoux} ( X; 1 - \varepsilon, \rho )$ finite. Let $\mu_X$ be a doubling measure on $( X, d_X )$, whose doubling constant is denoted by $C$. Fix a positive numerical parameter $\varepsilon < 1$ arbitrarily. 
For some $\rho > 0$ such that $\mu_X ( B ( x, 2 \rho ) ) \le 1 - \varepsilon$ for all $x \in X$,
combining the context regarding the doubling constant with \eqref{eq:expansion coefficient by Ledoux property}, we deduce that
\begin{equation}
\Expansion_{\Ledoux} ( X; 1 - \varepsilon, \rho ) 
\le 
\frac{\mu_X ( B ( x, 2 \rho ) )}{\mu_X ( B ( x, \rho ) )}
\le 
C,
\label{eq:sufficient condition for finiteness of Ledoux}
\end{equation}
as claimed.
\end{prop}

Our next claim is to discuss the sufficient condition on $\Expansion_{\Ledoux} ( X; 1 - \varepsilon, \rho ) > 1$, for which the \textit{Poincar{\'e} inequality} (see \eqref{eq:Poincare Inequality} below) on metric measure spaces plays a crucial role.
%
%
Before stating the condition to be observed, we give the following two quantities: For all local Lipschitz real-valued functions $f$ on $( X, d_X )$,
\begin{equation*}
\mathrm{Var}_{\mu_{X}} (f) := \int_X f^2\,d \mu_{X} - \left( \int_X f\,d \mu_{X} \right)^2
\end{equation*}
and
\begin{equation*}
|\nabla f| (x) := \limsup_{y \to x} \frac{|f(x) - f(y)|}{d_{X} (x, y)},
\end{equation*}
which we call the \textit{variance} of $f$ with respect to $\mu_X$ and the \textit{length of the gradient} of $f$ at the point $x \in X$, respectively.

The Poincar{\'e} inequality to be stated ensures that $\Expansion_{\Ledoux} ( X; 1 - \varepsilon, \rho ) > 1$:
\begin{thm}[Corollary~3.2 of \cite{RefLedoux2001}]
\label{thm:Poincare}
Let $( X, d_X, \mu_X )$ satisfy the Poincar{\'e} inequality with respect to the generalized length of gradient $|\nabla f|$ for all locally Lipschitz real-valued functions $f$ on $(X, d_X)$:
\begin{equation}
\Var_{\mu_{X}} (f) \le C \int_X |\nabla f|^2\,d \mu_{X}
\label{eq:Poincare Inequality}
\end{equation}
for some universal numerical constant $C > 0$.
Then we have $\Expansion_{\Ledoux} ( X; 1 - \varepsilon, \rho ) > 1$. 
\end{thm}

\begin{proof}
We omit the proof because its scenario runs almost parallel to that of Theorem~3.1 of \cite{RefLedoux2001}. 
\end{proof}



\end{document}